\documentclass[12pt,reqno]{amsart}

\usepackage{amsmath,amssymb,amsthm}

\usepackage[
  colorlinks=true,
  bookmarks=true,
  bookmarksnumbered=true,
  bookmarkstype=toc,
  linkcolor=blue,
  urlcolor=blue,
  citecolor=blue
]{hyperref}

\theoremstyle{plain}
\newtheorem{thm}{Theorem}[section]
\newtheorem{lem}[thm]{Lemma}
\newtheorem{prop}[thm]{Proposition}

\theoremstyle{definition}
\newtheorem{dfn}[thm]{Definition}

\theoremstyle{remark}
\newtheorem{rem}{Remark}[section]

\numberwithin{equation}{section}

\DeclareMathOperator{\dist}{dist}
\newcommand{\R}{\mathbb{R}}
\newcommand{\M}{\mathcal M}
\renewcommand{\phi}{\varphi}

\begin{document}

\title[Logistic equations with sanctuary regions and nonlocal diffusion]{
Fully Nonlinear Logistic Equations with Sanctuary Regions and Nonlocal Diffusion
}

\author[G. Galise]{Giulio Galise
}
\address[G. Galise]{
Dipartimento di Matematica Guido Castelnuovo, Sapienza Universit\`a di Roma, Piazzale Aldo Moro 5, Rome, Italy
}
\email{galise@mat.uniroma1.it}

\author[S. Kitano]{Shuhei Kitano
}
\address[S. Kitano]{
Research Institute for Science and Engineering, Waseda University, 3-4-1 Okubo, Shinjuku, Tokyo, 169-8555, Japan
}
\email{sk.koryo@moegi.waseda.jp}

\begin{abstract}
We study positive solutions of $Iu+\mu u=k(x)u^p$ in a bounded domain,
where $p>1$, the nonnegative coefficient $k$ may vanish on a sanctuary region,
and $I$ is a fully nonlinear nonlocal operator uniformly elliptic in the
Caffarelli--Silvestre sense.  We characterize the sharp existence interval in
terms of positive principal half-eigenvalues and prove uniqueness.  At the
lower endpoint, the solutions vanish uniformly and their normalized profiles
converge to the principal eigenfunction.  If $k$ vanishes on a sanctuary
$\Omega_0$, then, as $\mu\uparrow\lambda(\Omega_0)$, the solutions blow up
locally uniformly in $\Omega$.  Writing
$\phi_0=\phi(\cdot,\Omega_0)$ for the normalized positive principal
eigenfunction in $\Omega_0$, we prove
\[
\frac{u_\mu}{\|u_\mu\|_{L^\infty(\Omega)}}\to\phi_0,
\qquad
\frac{u_\mu}{\|u_\mu\|_{L^\infty(\Omega)}^{1/p}}
\to\left(\frac{I\phi_0}{k}\right)^{1/p},
\]
uniformly in $\R^n$ and locally uniformly in
$\Omega\setminus\overline{\Omega_0}$, respectively.  Thus the upper-endpoint
blow-up occurs on two distinct scales.
\end{abstract}


\maketitle

\section{Introduction}
We study positive solutions of the stationary nonlocal logistic problem
\begin{equation}\label{e1}
\left\{
\begin{aligned}
    Iu+\mu u&=k(x)u^p&&\text{in }\Omega,\\
    u&=0&&\text{in }\R^n\setminus\Omega,
\end{aligned}
\right.
\end{equation}
where $\Omega\subset\R^n$ is a bounded $C^2$ domain, $\mu\in\R$ is a
parameter, $p>1$, and $k\in C(\overline\Omega)$ is nonnegative and not
identically zero.  The operator $I$ is fully nonlinear and nonlocal, and is
uniformly elliptic in the sense of Caffarelli--Silvestre.  Precise assumptions
are given in Section~\ref{s2}.  Throughout the paper, a solution of
\eqref{e1} means a bounded continuous viscosity solution in the sense of
Definition~\ref{def-viscosity}.  The model examples are the negative fractional
Laplacian
\[
-(-\Delta)^s u(x)=\frac{c_{n,s}}{2}\,\operatorname{P.V.}\!\int_{\R^n}
\frac{\delta(u,x,y)}{|y|^{n+2s}}\,dy,
\qquad
\delta(u,x,y):=u(x+y)+u(x-y)-2u(x),
\]
where $c_{n,s}>0$ is the usual normalization constant.  More generally, we
consider the Isaacs operator
\begin{equation}\label{I}
Iu(x):=\inf_{a\in A}\sup_{b\in B}
\int_{\R^n}\delta(u,x,y)K_{a,b}(y)\,dy,
\end{equation}
where $A$ and $B$ are nonempty index sets and the symmetric kernels
$K_{a,b}$ are comparable with
$|y|^{-n-2s}$.

The main feature of the problem is that the absorption coefficient $k$ is
allowed to vanish on a sanctuary region.  We distinguish the non-sanctuary
case
\begin{equation*}\tag{K1}\label{K1}
    k(x)>0\qquad\text{for every }x\in\Omega,
\end{equation*}
from the sanctuary case: there exists a bounded $C^2$ domain
$\Omega_0\subsetneq\Omega$, not necessarily compactly contained in $\Omega$,
such that
\begin{equation*}\tag{K2}\label{K2}
\left\{
\begin{aligned}
    k&=0&&\text{on }\overline{\Omega_0}\cap\Omega,\\
    k&>0&&\text{in }\Omega\setminus\overline{\Omega_0}.
\end{aligned}
\right.
\end{equation*}
For a bounded domain $D$, we denote by
$\lambda(D)=\lambda_1^+(I,D)$ the positive principal half-eigenvalue of $I$
in $D$, and by $\phi(\cdot,D)$ the corresponding positive eigenfunction,
normalized by $\|\phi(\cdot,D)\|_{L^\infty(\R^n)}=1$.  Their precise
definitions and basic properties are recalled in Section~\ref{s2}.  Under
\eqref{K2}, one has
$\lambda(\Omega)<\lambda(\Omega_0)$; see
Proposition~\ref{lem-strict-domain-monotonicity} below.

Our first result gives the sharp existence interval.
\begin{thm}\label{thm11}
    Let $I$ satisfy \textup{(H1)}.  If $k$ satisfies \eqref{K1}, then \eqref{e1} admits a
    unique positive solution if and only if
    $\lambda(\Omega)<\mu$.  If $k$ satisfies \eqref{K2}, then \eqref{e1}
    admits a unique positive solution if and only if
    \[
       \lambda(\Omega)<\mu<\lambda(\Omega_0).
    \]
\end{thm}

We next describe the behavior at the endpoints of this interval.  At the
lower endpoint, the solution vanishes and its normalized profile converges to
the principal eigenfunction in $\Omega$.
\begin{thm}\label{thm-lower-asymptotics}
    Let $I$ satisfy \textup{(H1)}, and suppose that either \eqref{K1} or
    \eqref{K2} holds.  As $\mu\downarrow\lambda(\Omega)$, the unique
    positive solution $u_\mu$ of \eqref{e1}
    converges to $0$ uniformly in $\R^n$.  Moreover,
    \[
       \frac{u_\mu}{\|u_\mu\|_{L^\infty(\Omega)}}
       \longrightarrow \phi(\cdot,\Omega)
    \]
    uniformly in $\R^n$.
\end{thm}

At the upper endpoint, a two-scale nonlocal profile emerges.
\begin{thm}\label{thm-upper-asymptotics}
    Let $I$ satisfy \textup{(H1)}, suppose that \eqref{K2} holds, and set
    $\phi_0:=\phi(\cdot,\Omega_0)$.  As
    $\mu\uparrow\lambda(\Omega_0)$, the unique positive solution $u_\mu$
    of \eqref{e1}
    diverges to $+\infty$ locally uniformly in $\Omega$.  Moreover,
    \[
       \frac{u_\mu}{\|u_\mu\|_{L^\infty(\Omega)}}
       \longrightarrow \phi_0
    \]
    uniformly in $\R^n$, while
    \[
       \frac{u_\mu(x)}{\|u_\mu\|_{L^\infty(\Omega)}^{1/p}}
       \longrightarrow
       \left(\frac{I\phi_0(x)}{k(x)}\right)^{1/p}
    \]
    locally uniformly in $\Omega\setminus\overline{\Omega_0}$.
\end{thm}
For $x\in\Omega\setminus\overline{\Omega_0}$, the quantity
$I\phi_0(x)$ is classically well defined; moreover, it is continuous and
strictly positive there.  These facts are verified in the proof of the final
assertion below, so the second profile in Theorem~\ref{thm-upper-asymptotics}
is well defined.

Thus the solution blows up throughout $\Omega$, but on two distinct scales.
On the sanctuary, the dominant scale is
$\|u_\mu\|_{L^\infty(\Omega)}$, and the normalized profile converges to the
principal eigenfunction of $\Omega_0$.  Outside the sanctuary this normalized
profile vanishes, but a second, smaller scale survives: the solution is of
order $\|u_\mu\|_{L^\infty(\Omega)}^{1/p}$ and has the explicit limit displayed
above.  This two-scale structure reflects the long-range coupling of the
nonlocal diffusion between $\Omega_0$ and its complement.

We now place these results in the context of the existing literature.  For the
local problem
\[
\Delta u+\mu u=k(x)u^p\quad\text{in }\Omega,
\qquad u=0\quad\text{on }\partial\Omega,
\]
the study is related to the classical work of Kazdan--Warner
\cite{kazdan1975existence} on prescribed scalar curvature.  For the local
problem, the case where $k$ vanishes on a nonempty region was studied by Ouyang
\cite{ouyang1991positive,ouyang1992positive} and del Pino
\cite{del1994positive}; see also del Pino--Felmer \cite{del1995multiple}.
Du--Huang \cite{du1999blow} developed the upper-endpoint
blow-up theory for the local problem, relating the limiting behavior to an
associated boundary blow-up problem outside the sanctuary.  In a different singular perturbation setting, Li--L\'opez-G\'omez--Sun
\cite{li2021sharpblowup} showed that spatial degeneracy can produce
distinct blow-up rates in the degenerate region and its complement.  Recently,
Birindelli--Galise--Leoni
\cite{birindelli2026fully} extended this theory to fully
nonlinear local operators and allowed the operator to depend on $x$.  By
contrast, we only treat the $x$-independent setting; see
Remark~\ref{rem-xdependence} below.

The study of nonlocal logistic equations with a sanctuary region was initiated by
Garc\'{i}a-Meli\'{a}n--Rossi \cite{garcia2009logistic};
for earlier works on nonlocal logistic equations with a strictly positive nonlinear
coefficient, see the references therein.  Their model is
\[
(J*u-u)+\mu u=k(x)u^p\quad\text{in }\Omega,
\qquad u=0\quad\text{in }\R^n\setminus\Omega,
\]
where $J$ is smooth, compactly supported, nonnegative, and has unit mass.
They obtained the optimal existence interval and a detailed description of the
upper-endpoint profile.  Writing $M_\mu=\|u_\mu\|_{L^\infty(\Omega)}$, they
showed that $u_\mu/M_\mu$ converges on the sanctuary to its principal
eigenfunction.  Because $J$ is compactly supported, the complement of the
sanctuary decomposes into successive regions of influence, and on the $m$th
region the corresponding scale is $M_\mu^{1/p^m}$.  In particular, the first
exterior profile is obtained from the convolution of the sanctuary
eigenfunction with $J$.  This provides a useful contrast with the present
setting: the kernels in \eqref{I} have full support and are bounded below by a
multiple of $|y|^{-n-2s}$, so every point outside the sanctuary interacts
directly with it.  Accordingly, the hierarchy of scales collapses here to the
two scales $M_\mu$ and $M_\mu^{1/p}$.

Coville \cite{coville2015nonlocal} subsequently studied
more general heterogeneous nonlocal refuge models with continuous kernels $K(x,y)$.
Removing the compact-support assumption on the convolution kernel in
\cite{garcia2009logistic}, Sun--You--Qiao
\cite{sun2021degeneracy} identified the normalized upper-endpoint
profiles at the $M_\lambda$- and $M_\lambda^{1/p}$-scales.

Related upper-endpoint blow-up phenomena for logistic equations in mixed
local--nonlocal settings were studied in
\cite{cintra2026asymptotic,cintra2025elliptic}.
Nonlocal logistic equations without a sanctuary have also been studied for broader classes of singular L\'evy operators; see, for example, Biswas--Modasiya \cite{biswas2022harvesting} and the references therein.

For the fractional Laplacian with Dirichlet condition, Klimsiak
\cite{klimsiak2023asymptotics} characterized the optimal eigenvalue interval for
positive stationary solutions in the presence of a refuge. The same paper also
studies a different asymptotic regime: as the reaction exponent tends to
infinity, the solutions converge to an obstacle-problem limit.
More recently, Constantin--Santos--Warnault
\cite{constantin2025logistic} studied elliptic and
parabolic logistic problems driven by the fractional $p$-Laplacian. In the
linear-diffusion case, their elliptic problem contains the fractional-Laplacian
instance of \eqref{e1}, and they obtained the same optimal existence interval
and the vanishing and blow-up behavior at its endpoints. In contrast with the
probabilistic approach of \cite{klimsiak2023asymptotics} and the variational
weak-solution framework of \cite{constantin2025logistic},
our approach is nonvariational and relies on maximum principles and viscosity
solution theory, allowing us to treat fully nonlinear nonlocal operators.

As noted above, normalized upper-endpoint profiles have been obtained for integrable convolution kernels, notably in \cite{garcia2009logistic,sun2021degeneracy}. To the best of our knowledge, however, the two profile limits in Theorem~1.3 are new even for the fractional Laplacian.

The paper is organized as follows.  Section~\ref{s2} develops the viscosity
framework, the comparison principle, and the principal-eigenvalue tools used
throughout the paper.  Section~3 constructs sub- and supersolutions and proves
Theorem~\ref{thm11} by Perron's method.  Section~4 is devoted to the endpoint
asymptotics and, in particular, to the two-scale upper-endpoint profile.
To keep the paper self-contained, the appendices establish the difference
inequality for sub- and supersolutions and carry out the Perron construction
for fully nonlinear nonlocal equations with nonlinear zeroth-order terms.

\section{Preliminaries and comparison principles}\label{s2}
We collect here the analytic tools used in the proof of the main theorem.

\subsection{Viscosity framework and difference inequalities}
For an open set $D\subset\R^n$, set
\[
d_D(x):=\dist(x,\R^n\setminus D),
\qquad
D_\theta:=\{x\in D:d_D(x)<\theta\},\quad \theta>0.
\]
For $t\in\R$, set $t^+:=\max\{t,0\}$ and $t^-:=\max\{-t,0\}$.
We impose the following condition on the nonlocal operator $I$:
\begin{enumerate}
    \item[(H1)] There exist constants $s\in(0,1)$ and $0<\lambda\leq\Lambda$ such that $I$ is of the form \eqref{I}, where $\{K_{a,b}\}_{a\in A,b\in B}$ is a family of measurable symmetric kernels satisfying
\[
\frac{\lambda}{|y|^{n+2s}}\leq K_{a,b}(y)\leq \frac{\Lambda}{|y|^{n+2s}}
\]
for every $a\in A$, $b\in B$, and $y\in\R^n\setminus\{0\}$.
\end{enumerate}
Throughout the remainder of the paper, \textup{(H1)} is assumed without further mention.
The Pucci extremal operators $\M^+$ and $\M^-$ are defined by
 \begin{align*}
     \M^+u(x)&=\int_{\R^n}\frac{\Lambda\delta(u,x,y)^+-\lambda\delta(u,x,y)^-}{|y|^{n+2s}}\,dy,\\
     \M^-u(x)&=\int_{\R^n}\frac{\lambda\delta(u,x,y)^+-\Lambda\delta(u,x,y)^-}{|y|^{n+2s}}\,dy.
 \end{align*}
 Since the kernels in (H1) are independent of $x$, the operator $I$ is
 translation invariant.  The operator is also positively $1$-homogeneous and
 invariant under addition of constants.  Moreover,
 \[
 \M^-(u-v)(x)\leq Iu(x)-Iv(x)\leq \M^+(u-v)(x)
 \]
 for all bounded functions $u,v$ that are $C^2$ in a neighborhood of $x$.

\begin{rem}\label{rem-xdependence}
The $x$-independence in \textup{(H1)} enters the proof of Lemma~\ref{lem-difference}
through the translation-invariant sup/inf-convolution argument of
\cite{caffarelli2009regularity} and is also used in the principal-eigenvalue
theory invoked from \cite{quaas2020principal}.  If a corresponding difference inequality for sub- and supersolutions and principal-eigenvalue results were available for an $x$-dependent uniformly
elliptic Isaacs class, the arguments of Sections~3 and~4 would apply in that
setting as well.
\end{rem}

 We write $USC(D)$ and $LSC(D)$ for the upper and lower
 semicontinuous functions on a set $D$, respectively.  We next recall the
 definition of viscosity solutions.
 \begin{dfn}\label{def-viscosity}
     Let $J\subset\R$ be an interval and let $F\in C(\Omega\times J)$.
     We say that $u\in USC(\Omega)\cap L^\infty(\R^n)$, taking values in
     $J$ in $\Omega$, is a viscosity subsolution of
     \[
        Iw=F(x,w)\quad\text{in }\Omega
     \]
     if, for every $x_0\in\Omega$, every neighborhood $V\Subset\Omega$ of
     $x_0$, and every $\zeta\in C^2(\overline V)$ touching $u$ from above at
     $x_0$ in $V$, the patched function
     \[
     z(y)=
     \begin{cases}
         \zeta(y),&y\in V,\\
         u(y),&y\in\R^n\setminus V
     \end{cases}
     \]
     satisfies
     \[
        Iz(x_0)\geq F(x_0,u(x_0)).
     \]
     A viscosity supersolution is defined analogously, with
     $v\in LSC(\Omega)\cap L^\infty(\R^n)$, a lower test function, and the
     reverse inequality.  A continuous function is a viscosity solution if it
     is both a subsolution and a supersolution.

     For \eqref{e1}, sub- and supersolutions are always assumed nonnegative in $\Omega$;
     the exterior Dirichlet condition is understood in the ordered sense
     $u\leq0$ for subsolutions and $v\geq0$ for supersolutions in
     $\R^n\setminus\Omega$.
 \end{dfn}

Throughout the paper, unless explicitly stated otherwise, equations and
inequalities involving semicontinuous functions are understood in the
viscosity sense of Definition~\ref{def-viscosity}.

The following lemma is a zeroth-order extension of the difference theorem in
\cite[Theorem~5.9]{caffarelli2009regularity}, allowing the right-hand side to
depend continuously on the value of the solution.
\begin{lem}[Difference inequality with a continuous majorant]\label{lem-difference}
    Let $J\subset\mathbb R$ be a compact interval, let
    $F,G\in C(\overline\Omega\times J)$, and let
    $u\in USC(\Omega)\cap L^\infty(\R^n)$ and
    $v\in LSC(\Omega)\cap L^\infty(\R^n)$ take values in $J$ in $\Omega$.
    Assume that $u$ and $v$ are, respectively, a viscosity subsolution and a
    viscosity supersolution of
    \[
       Iu=F(x,u)\quad\text{in }\Omega,
       \qquad
       Iv=G(x,v)\quad\text{in }\Omega.
    \]
    Suppose in addition that there is
    $H\in C(\overline\Omega\times\mathbb R)$ such that
    \begin{equation}\label{eq-continuous-majorant}
       G(x,b)-F(x,a)\leq H(x,b-a)
       \qquad (x\in\overline\Omega,\ a,b\in J).
    \end{equation}
    Then $w:=v-u\in LSC(\Omega)\cap L^\infty(\R^n)$ is a viscosity
    supersolution of
    \begin{equation}\label{eq-difference-semicont}
       \M^-w=H(x,w)\qquad\text{in }\Omega.
    \end{equation}
\end{lem}
The proof is given in Appendix~\ref{app-difference}.

We shall use the immediate consequence that a bounded continuous function
cannot solve $Iw=f(x)$ and $Iw=g(x)$ in the same domain unless $f=g$; this
follows by applying Lemma~\ref{lem-difference} in the two orders.  We refer to
this below as uniqueness of the continuous forcing.

\subsection{Maximum and comparison principles}
We first record a minimum principle on sets of small measure.
\begin{lem}[Small-measure minimum principle]\label{lem-small-measure}
    For every $\mu\in\R$, there exists
    $c_1=c_1(n,s,\lambda,\mu)>0$ with the following property.  Let $D$ be a
    bounded open set with $|D|\leq c_1$, and let
    $q\in LSC(\overline D)\cap L^\infty(\R^n)$ satisfy
    \[
       q\geq0\quad\text{in }\R^n\setminus D.
    \]
    Assume that $q$ is a viscosity supersolution of
    \[
       \M^-q+\mu q=0\quad\text{in }D.
    \]
    Then $q\geq0$ in $\R^n$.
\end{lem}
\begin{proof}
    Suppose, to the contrary, that $q<0$ somewhere.  Then
    $m:=\min_{\overline D}q<0$, attained at $x_0\in D$.  Fix
    $r<\dist(x_0,\partial D)$ and use the constant function $m$ as a lower
    test for $q$ in $B_r(x_0)$.  The corresponding patched function
    \[
       z(x)=
       \begin{cases}
          m,&x\in B_r(x_0),\\
          q(x),&x\in\R^n\setminus B_r(x_0)
       \end{cases}
    \]
    satisfies $z\geq m=z(x_0)$ in $\R^n$.  Hence every second difference of
    $z$ at $x_0$ is nonnegative, and therefore
    \[
       \M^-z(x_0)
       =2\lambda\int_{\R^n}
          \frac{z(y)-m}{|y-x_0|^{n+2s}}\,dy
       \geq2\lambda(-m)
          \int_{\R^n\setminus D}\frac{dy}{|y-x_0|^{n+2s}},
    \]
    because $z=q\geq0$ on $\R^n\setminus D$.  By
    \cite[Lemma~A.1]{savin2014density}, there exists
    $c_*=c_*(n,s)>0$ such that
    \[
       \int_{\R^n\setminus D}\frac{dy}{|y-x_0|^{n+2s}}
       \geq c_*|D|^{-2s/n}.
    \]
    The viscosity inequality at $x_0$ gives
    \[
       0\geq \M^-z(x_0)+\mu m
       \geq(-m)\bigl(2\lambda c_*|D|^{-2s/n}-\mu\bigr).
    \]
    Choose $c_1>0$ so that
    $2\lambda c_*c_1^{-2s/n}>\mu^+$.  The right-hand side is then
    strictly positive, a contradiction.
\end{proof}

We use the following strong maximum principle, whose zero-order-free case is
\cite[Theorem~2.8]{quaas2020principal}.
\begin{lem}[Strong maximum principle]\label{lem-smp}
    Let $C\in\R$, and suppose that
    $w\in LSC(\Omega)\cap L^\infty(\R^n)$ satisfies
    $w\geq0$ in $\R^n$ and is a viscosity supersolution of
    \[
    \M^-w=Cw\quad\text{in }\Omega.
    \]
    Then either $w>0$ in $\Omega$ or
    $w=0$ in $\Omega$ and almost everywhere in $\R^n$.
\end{lem}
\begin{proof}
    Suppose that $w(x_0)=0$ for some $x_0\in\Omega$.  Fix
    $0<r<\dist(x_0,\R^n\setminus\Omega)$ and let $z_r=0$ in
    $B_r(x_0)$ and $z_r=w$ in $\R^n\setminus B_r(x_0)$.  Since the constant
    zero touches $w$ from below at $x_0$, the viscosity inequality gives
    \[
       0\geq \M^-z_r(x_0)
       =2\lambda\int_{\R^n\setminus B_r(x_0)}
         \frac{w(y)}{|y-x_0|^{n+2s}}\,dy\geq0.
    \]
    Hence $w=0$ almost everywhere outside $B_r(x_0)$.  Letting $r\downarrow0$
    gives $w=0$ almost everywhere in $\R^n$, and lower semicontinuity and
    nonnegativity then imply $w\equiv0$ in $\Omega$.
\end{proof}
\begin{lem}[Comparison principle]\label{lem4}
    Let $\mu\in\R$.
    If $u\in USC(\overline{\Omega})\cap L^\infty(\R^n)$ and
    $v\in LSC(\overline{\Omega})\cap L^\infty(\R^n)$ are, respectively,
    a nonnegative viscosity subsolution and a positive viscosity supersolution
    of \eqref{e1}, then $u\leq v$ in $\Omega$.
\end{lem}
\begin{proof}
    For $0<t\leq1$, define
    \[
       F_t(x,r):=t^{1-p}k(x)r^p-\mu r,\qquad r\geq0.
    \]
    By the positive homogeneity of $I$, the function $tu$ is a viscosity
    subsolution of $Iw=F_t(x,w)$ in $\Omega$, whereas $v$ is a
    supersolution with $t=1$.  Put
    \[
       R:=\max\{\|u\|_\infty,\|v\|_\infty\},\qquad
       C_R:=p\|k\|_\infty R^{p-1}.
    \]
    For $a,b\in[0,R]$ and $0<t\leq1$,
    \begin{align*}
       F_1(x,b)-F_t(x,a)
       &=k(x)\bigl(b^p-t^{1-p}a^p\bigr)-\mu(b-a)\\
       &\leq C_R(b-a)^+-\mu(b-a).
    \end{align*}
    Lemma~\ref{lem-difference}, with the continuous majorant
    $H(x,r)=C_Rr^+-\mu r$, therefore gives
    \begin{equation}\label{eq-logistic-difference}
       \M^-w_t+\mu w_t\leq C_R(w_t)^+
       \quad\text{in }\Omega,\qquad w_t:=v-tu.
    \end{equation}

    Choose a boundary layer $D=\Omega_\delta$ with
    $|D|\leq c_1(n,s,\lambda,\mu)$ and put $K=\Omega\setminus D$,
    where $K\Subset\Omega$.  Since $v>0$ and is lower semicontinuous,
    $tu\leq v$ on $K$ for all sufficiently small $t>0$.  Hence
    $w_t\geq0$ in $\R^n\setminus D$.  Define
    \[
       \widehat w_t(x):=
       \begin{cases}
          \min\{w_t(x),0\},&x\in D,\\
          0,&x\in\R^n\setminus D.
       \end{cases}
    \]
    Since $w_t\in LSC(\overline\Omega)$ and $w_t\geq0$ on $\partial D$,
    we have $\widehat w_t\in LSC(\R^n)\cap L^\infty(\R^n)$.  We claim that
    \[
       \M^-\widehat w_t+\mu\widehat w_t\leq0\quad\text{in }D.
    \]
    Indeed, at lower contact points where $\widehat w_t<0$, the claim
    follows from \eqref{eq-logistic-difference} and ellipticity, while at
    those where $\widehat w_t=0$, we have $\M^-\widehat w_t\leq0$.
    Lemma~\ref{lem-small-measure} now gives $\widehat w_t\geq0$, and thus
    $tu\leq v$ in $\Omega$ for all sufficiently small $t$.

    Let
    \[
       \tau=\sup\{t\in(0,1]:tu\leq v\text{ in }\Omega\}.
    \]
    Taking $t\to\tau$, we obtain $\tau u\leq v$ in $\Omega$.  If $\tau<1$, set
    $w=v-\tau u\geq0$.  By
    \eqref{eq-logistic-difference},
    \[
       \M^-w\leq(C_R-\mu)w\quad\text{in }\Omega
    \]
    Lemma~\ref{lem-smp} implies either
    $w>0$ in $\Omega$ or $w=0$ in $\Omega$ and almost everywhere in
    $\R^n$.

    We exclude the latter alternative.  If $w=0$, then $v=\tau u$ in
    $\Omega$ and almost everywhere in $\R^n$.  Since $u$ is upper
    semicontinuous and $v$ is lower semicontinuous, the identity $v=\tau u$
    makes $v$ continuous in $\Omega$.  Thus $v$ is both a subsolution with
    nonlinearity $F_\tau$ and a supersolution with nonlinearity $F_1$.
    Applying Lemma~\ref{lem-difference} to these two inequalities, we obtain
    \[
       0=\M^-0\leq F_1(x,v)-F_\tau(x,v),
    \]
    contradicting $\tau<1$, $v>0$, and $k\not\equiv0$.  Hence $w>0$ in
    $\Omega$.

    Thus $w$ has a positive minimum on $K$, and for some
    $\varepsilon\in(0,1-\tau)$ we have
    $v-(\tau+\varepsilon)u\geq0$ on $K$. Applying the preceding argument with
    $t=\tau+\varepsilon$, Lemma~\ref{lem-small-measure} yields
    \[
    v-(\tau+\varepsilon)u\geq0\quad\text{in }\Omega,
    \]
    contradicting the definition of $\tau$.
\end{proof}

\subsection{Principal eigenvalues and rigidity}
 The convention in \eqref{I} agrees with that of
 \cite{quaas2020principal} when the drift term there is set equal to zero.
 Although the main eigenvalue theorems are stated for $s\in(1/2,1)$ in the
 presence of a drift, the authors observe that in the zero-drift case the same
 conclusions hold for every $s\in(0,1)$.  Following that paper, we denote by
 $\lambda(\Omega)=\lambda_1^+(I,\Omega)$ the positive principal
 half-eigenvalue of $I$ in $\Omega$, namely,
 \[
 \begin{aligned}
 \lambda(\Omega)=\sup\bigl\{\eta\in\R:\;&\exists
 v\in C(\overline\Omega)\cap L^\infty(\R^n),\quad
 v>0\text{ in }\Omega,\\
 &v\geq0\text{ in }\R^n,\quad
 \text{$v$ is a viscosity supersolution of }
 Iv+\eta v=0\text{ in }\Omega\bigr\}.
 \end{aligned}
 \]
 This bounded formulation agrees with \cite{quaas2020principal}: bounded
 admissible functions belong to the weighted tail class used there, while the
 principal eigenfunction has zero exterior values and is bounded.  The
 following is the zero-drift specialization of
 \cite[Theorems~1.1 and~1.2]{quaas2020principal}.
 \begin{prop}[Principal eigenpair]\label{prop21}
     There exists a unique positive viscosity solution $\phi\in C(\R^n)$ of
     \[
     \left\{
     \begin{aligned}
         I\phi+\lambda(\Omega)\phi&=0&&\text{in }\Omega,\\
         \phi&=0&&\text{in }\R^n\setminus\Omega,
     \end{aligned}
     \right.
     \]
     such that $\|\phi\|_{L^\infty(\R^n)}=1$.
 \end{prop}
 In what follows, we denote this eigenfunction by $\phi(\cdot,\Omega)$.
The following proportionality principle is a semicontinuous version of the
sweeping argument in \cite[Theorem~5.1]{quaas2020principal}.
\begin{lem}[Proportionality principle]\label{lem-proportionality}
    Let $\eta\in\R$.  Suppose that
    $V\in LSC(\overline\Omega)\cap L^\infty(\R^n)$ and
    $U\in USC(\overline\Omega)\cap L^\infty(\R^n)$ satisfy
    \[
    \begin{cases}
        IV+\eta V\leq0 & \text{in }\Omega,\\
        V\geq0 & \text{in }\R^n,
    \end{cases}
    \qquad
    \begin{cases}
        IU+\eta U\geq0 & \text{in }\Omega,\\
        U\leq0 & \text{in }\R^n\setminus\Omega,\\
        U(x_0)>0 & \text{for some }x_0\in\Omega.
    \end{cases}
    \]
    Then there exists $\tau\geq0$ such that
    \[
    V=\tau U\quad\text{in }\Omega
    \quad\text{and almost everywhere in }\R^n.
    \]
\end{lem}
\begin{proof}
    By ellipticity,
    \[
       \M^-V+\eta V\leq0\quad\text{in }\Omega.
    \]
    If $V$ vanishes at some point of $\Omega$, Lemma~\ref{lem-smp} gives
    $V=0$ in $\Omega$ and almost everywhere in $\R^n$, so the conclusion
    holds with $\tau=0$.  Otherwise $V>0$ in $\Omega$.  Choose a compact
    set $K\Subset\Omega$ with $x_0\in K$ and
    $|\Omega\setminus K|\leq c_1(n,s,\lambda,\eta)$, where $c_1$ is the
    constant in Lemma~\ref{lem-small-measure}, and set
    \[
       \max_K\frac{U}{V}=\frac{1}{\tau}.
    \]
    By semicontinuity and the positivity of $V$ on $K$, the maximum is
    attained.  Put $q:=V-\tau U$.  Since $V\geq0$ in $\R^n$ and
    $U\leq0$ in $\R^n\setminus\Omega$, we have
    \[
       q\geq0\quad\text{in }(\R^n\setminus\Omega)\cup K.
    \]
    By Lemma~\ref{lem-difference},
    \[
       \M^-q+\eta q\leq0\quad\text{in }\Omega.
    \]
    Lemma~\ref{lem-small-measure}, applied in $\Omega\setminus K$, yields
    $q\geq0$ in $\R^n$.  Moreover, at a point $\bar x\in K$ where the
    maximum defining $\tau^{-1}$ is attained, $q(\bar x)=0$.  Therefore
    Lemma~\ref{lem-smp} gives $q=0$ in $\Omega$ and almost everywhere in
    $\R^n$.
\end{proof}

The proportionality principle gives the following rigidity statement at the
principal eigenvalue.
\begin{prop}[Rigidity at the principal eigenvalue]\label{prop}
    Suppose that $v\in LSC(\Omega)\cap L^\infty(\R^n)$ satisfies $v\geq0$
    in $\R^n$ and is a viscosity supersolution of
    \[
    Iv+\lambda(\Omega)v=0\quad\text{in }\Omega.
    \]
    Then there exists $\alpha_1\geq0$ such that
    $v=\alpha_1\phi(\cdot,\Omega)$ in $\Omega$ and almost everywhere in
    $\R^n$.

    Suppose that $u\in USC(\overline{\Omega})\cap L^\infty(\R^n)$ is a
    nonnegative viscosity subsolution of
    \[
    Iu+\lambda(\Omega)u=0\quad\text{in }\Omega,
    \]
    and that $u=0$ in $\R^n\setminus\Omega$.
    Then there exists $\alpha_2\geq0$ such that
    $u=\alpha_2\phi(\cdot,\Omega)$ in $\R^n$.
\end{prop}
\begin{proof}
    Write $\lambda_1=\lambda(\Omega)$ and $\phi=\phi(\cdot,\Omega)$.

    For the supersolution $v$, ellipticity gives
    \[
       \M^-v+\lambda_1v\leq0\quad\text{in }\Omega.
    \]
    Hence Lemma~\ref{lem-smp} shows that either
    $v=0$ in $\Omega$ and almost everywhere in $\R^n$, or $v>0$ in $\Omega$.
    The first case is immediate.  In the second case define
    \[
       V(x):=\begin{cases}
          v(x),&x\in\Omega,\\
          0,&x\in\R^n\setminus\Omega.
       \end{cases}
    \]
    Then $V\in LSC(\overline\Omega)\cap L^\infty(\R^n)$,
    $V\leq v$ in $\R^n$, and $V=v$ in $\Omega$.  Thus ellipticity implies
    that $V$ is a viscosity supersolution of
    \[
       IV+\lambda_1V\leq0\quad\text{in }\Omega.
    \]
    Applying Lemma~\ref{lem-proportionality} with $U=\phi$ gives
    $v=V=\alpha_1\phi$ in $\Omega$ for some $\alpha_1>0$.
    Set $w:=v-\alpha_1\phi$.  Then $w\geq0$ in $\R^n$ and $w=0$ in
    $\Omega$.  Lemma~\ref{lem-difference} gives
    \[
       \M^-w+\lambda_1w\leq0\quad\text{in }\Omega,
    \]
    and Lemma~\ref{lem-smp} yields $w=0$ almost everywhere in $\R^n$.

    For the subsolution $u$, the case $u\equiv0$ is immediate.  Otherwise
    $u(x_0)>0$ for some $x_0\in\Omega$, and
    Lemma~\ref{lem-proportionality}, applied with $U=u$ and $V=\phi$,
    gives $u=\alpha_2\phi$ in $\Omega$ and almost everywhere in $\R^n$ for
    some $\alpha_2>0$.  Since both functions vanish in
    $\R^n\setminus\Omega$, the equality holds in all of $\R^n$.
\end{proof}

\begin{prop}[Strict domain monotonicity]\label{lem-strict-domain-monotonicity}
    Let $D_1\subset D_2$ be bounded $C^2$ domains such that
    $|D_2\setminus D_1|>0$.  Then
    \[
       \lambda(D_2)<\lambda(D_1).
    \]
\end{prop}
\begin{proof}
    Let $\phi_2:=\phi(\cdot,D_2)$.  Since $\phi_2>0$ in $D_1$,
    $\phi_2\geq0$ in $\R^n\setminus D_1$, and
    \[
       I\phi_2+\lambda(D_2)\phi_2=0
       \quad\text{in }D_1,
    \]
    the definition of the positive principal eigenvalue in $D_1$ gives
    $\lambda(D_2)\leq\lambda(D_1)$.  Suppose that equality holds.
    Applying Proposition~\ref{prop}, with $D_1$ in place of $\Omega$, to
    $\phi_2$ yields
    \[
       \phi_2=\alpha\phi(\cdot,D_1)
       \quad\text{in }D_1\text{ and almost everywhere in }\R^n
    \]
    for some $\alpha>0$.  Since $\phi(\cdot,D_1)=0$ in
    $\R^n\setminus D_1$, this implies that $\phi_2=0$ almost everywhere
    in $D_2\setminus D_1$, contradicting $\phi_2>0$ in $D_2$ and
    $|D_2\setminus D_1|>0$.
\end{proof}

\subsection{Regularity and domain perturbation}
We shall use the following global H\"older estimate.
\begin{lem}[Global H\"older estimate]\label{lem3}
    Let $\Omega$ be a bounded $C^2$ domain in $\R^n$ and let
    $f\in L^\infty(\Omega)$.  Then there exist $C_2>0$ and
    $\beta\in(0,1)$, depending only on $n,s,\lambda,\Lambda$, and $\Omega$, such that every viscosity solution
    $u\in C(\R^n)$ of
    \[
       Iu=f\quad\text{in }\Omega,
       \qquad
       u=0\quad\text{in }\R^n\setminus\Omega,
    \]
    satisfies
    \[
       \|u\|_{C^\beta(\R^n)}
       \leq C_2\bigl(\|u\|_{L^\infty(\R^n)}+\|f\|_{L^\infty(\Omega)}\bigr).
    \]
\end{lem}
This was proved by Quaas--Salort--Xia
\cite[Theorem~3.6]{quaas2020principal} for $s\in(1/2,1)$ and by
Fern\'andez-Real--Ros-Oton
\cite[Corollary~3.3.8]{fernandezreal2024integro}.

We shall also use the continuity of the principal eigenvalue under smooth
outer perturbations of the domain.
\begin{lem}[Outer continuity of the principal eigenvalue]\label{lem-domain-continuity}
    Let $D$ be a bounded $C^2$ domain.  There exist bounded $C^2$ domains
    $D_m$ such that
    \[
    \overline D\Subset D_{m+1}\Subset D_m,
    \qquad
    \bigcap_{m=1}^\infty D_m=\overline D,
    \qquad
    \lambda(D_m)\longrightarrow\lambda(D).
    \]
\end{lem}
\begin{proof}
    Let $\rho_D$ be the signed distance to $\partial D$, chosen negative in
    $D$.  Since $\rho_D$ is $C^2$ in a tubular neighborhood of $\partial D$,
    choose a strictly decreasing sequence $\varepsilon_m\downarrow0$ sufficiently
    small and set
    \[
       D_m:=\{x\in\R^n:\rho_D(x)<\varepsilon_m\}.
    \]
    Then $\overline D\Subset D_{m+1}\Subset D_m$,
    $\bigcap_mD_m=\overline D$, and the domains $D_m$ have a uniform bounded
    $C^2$ character.

    Domain monotonicity gives $\lambda(D_m)\uparrow\ell\leq\lambda(D)$.
    Let $\phi_m$ be the corresponding nonnegative eigenfunctions normalized
    by $\|\phi_m\|_{L^\infty(\R^n)}=1$.  Lemma~\ref{lem3}, the uniform
    $C^2$ geometry, and the uniform bound on $\lambda(D_m)$ give a common
    H\"older modulus.  Since all $\phi_m$ vanish outside $D_1$, after passing
    to a subsequence,
    \[
       \phi_m\longrightarrow\phi\quad\text{uniformly in }\R^n.
    \]
    The domain exhaustion gives $\phi=0$ in $\R^n\setminus D$ and
    $\|\phi\|_{L^\infty(\R^n)}=1$.  By the stability of viscosity solutions
    \cite[Corollary~4.7]{caffarelli2009regularity},
    \[
       I\phi+\ell\phi=0\quad\text{in }D.
    \]
    Since $\ell\leq\lambda(D)$, $\phi$ is a nonnegative subsolution of the
    principal eigenvalue equation in $D$.  Proposition~\ref{prop} and the
    normalization give $\phi=\phi(\cdot,D)$.  Applying Lemma~\ref{lem-difference} in the two orders to
    \[
       I\phi=-\ell\phi
       \qquad\text{and}\qquad
       I\phi=-\lambda(D)\phi
       \quad\text{in }D,
    \]
    we obtain
    \[
       (\ell-\lambda(D))\phi=0\quad\text{in }D.
    \]
    Since $\phi>0$ in $D$, it follows that $\ell=\lambda(D)$.
\end{proof}

\section{Existence and nonexistence of positive solutions}
In this section, we prove Theorem~\ref{thm11}.  We first construct ordered
sub- and supersolutions and then close the argument by Perron's method and the
comparison principle.
\subsection{Sub- and supersolutions}
We first record the subsolution supplied by the principal eigenfunction.
\begin{lem}[Positive subsolution]\label{lem-positive-subsolution}
    If $\mu>\lambda(\Omega)$, then
    \[
       \underline w=c_3\phi(\cdot,\Omega),\qquad
       c_3=\left(
       \frac{\mu-\lambda(\Omega)}{\|k\|_{L^\infty(\Omega)}}
       \right)^{1/(p-1)},
    \]
    is a positive viscosity subsolution of \eqref{e1}.
\end{lem}
\begin{proof}
    The exterior condition is immediate, and
    \[
       I\underline w+\mu\underline w
       =(\mu-\lambda(\Omega))\underline w
       \geq k\underline w^p
       \quad\text{in }\Omega,
    \]
    where we used $0<\phi(\cdot,\Omega)\leq1$ and the definition of
    $c_3$.
\end{proof}

We next construct a supersolution.  The following auxiliary function allows us
to control the nonlocal operator away from the sanctuary.
\begin{lem}[Auxiliary positive function]\label{lem32}
    Let $\Omega_1$ be a bounded $C^2$ domain, and let $\mu<\lambda(\Omega_1)$.
    Then there exists $\psi\in C(\R^n)\cap L^\infty(\R^n)$ with the following properties:
    \begin{enumerate}
        \item[(I)] $\psi$ is a viscosity supersolution of
        \[
        I\psi+\mu\psi=0\quad\text{in }\Omega_1;
        \]
        \item[(II)] for some constant $C_4>0$, $\psi$ is a viscosity supersolution of
        \[
        I\psi+\mu\psi=C_4\quad\text{in }\R^n;
        \]
        \item[(III)] $\psi$ is strictly positive in $\R^n$; more precisely, $\inf_{\R^n}\psi=:\theta>0$.
    \end{enumerate}
\end{lem}
\begin{proof}
    By Lemma~\ref{lem-domain-continuity}, choose a $C^2$ domain
    $\Omega_2$ such that $\Omega_1\Subset\Omega_2$ and
    $\mu<\lambda(\Omega_2)$.
    Let $\phi(\cdot,\Omega_2):\R^n\to\R$ be the normalized eigenfunction in $\Omega_2$ given by Proposition~\ref{prop21}.
    Choose $\eta_1\in C^\infty_c(\Omega_2)$ that is positive on $\overline{\Omega_1}$ and nonnegative in $\Omega_2$.
    By multiplying $\eta_1$ by a large constant if necessary, we may assume $\eta_1>\phi(\cdot,\Omega_2)$ in $\overline{\Omega_1}$.
    Set
    \[
    \widetilde{\phi}=\min\{\phi(\cdot,\Omega_2),\eta_1\}.
    \]
    The strict inequality and compactness give a neighborhood $U$ of
    $\overline{\Omega_1}$ on which $\eta_1>\phi(\cdot,\Omega_2)$; hence
    $\widetilde\phi=\phi(\cdot,\Omega_2)$ in $U$.  Moreover,
    $\widetilde\phi\leq\phi(\cdot,\Omega_2)$ in $\R^n$.  Consequently, if
    a lower test function touches $\widetilde\phi$ at $x\in\Omega_1$, the
    corresponding patched function is bounded above by the one obtained by
    patching the same test function to $\phi(\cdot,\Omega_2)$.  The
    ellipticity of $I$ and the eigenfunction equation therefore give
    \[
    I\widetilde{\phi}+\mu\widetilde{\phi}
    \leq(\mu-\lambda(\Omega_2))\widetilde{\phi}
    \leq-\varepsilon\quad\text{in }\Omega_1,
    \]
    where
    \[
    \varepsilon=(\lambda(\Omega_2)-\mu)
       \min_{\overline{\Omega_1}}\phi(\cdot,\Omega_2)>0.
    \]
    Since $\eta_1$ has compact support in $\Omega_2$, the function
    $\widetilde{\phi}$ vanishes in a neighborhood of $\partial\Omega_2$
    and extends continuously by zero to $\R^n\setminus\Omega_2$.

    We next justify the global upper bound.  Since
    $\eta_1\in C^\infty_c(\R^n)$, the quantity
    $I\eta_1+\mu\eta_1$ is classically defined and bounded in $\R^n$.
    We next obtain a global bounded-defect estimate.  In $\Omega_1$ the
    preceding inequality already gives
    $I\widetilde\phi+\mu\widetilde\phi\leq0$.  Let
    $x\in\R^n\setminus\Omega_1$ and let a lower test function touch
    $\widetilde\phi$ at $x$.  If
    $\widetilde\phi(x)=\eta_1(x)$, then it also touches $\eta_1$ from
    below, and the patched function for $\widetilde\phi$ is bounded above
    by the corresponding one for $\eta_1$.  Thus
    \[
    I\widetilde\phi(x)+\mu\widetilde\phi(x)
       \leq I\eta_1(x)+\mu\eta_1(x).
    \]
    If $\widetilde\phi(x)<\eta_1(x)$, then $x\in\Omega_2$ and the same
    comparison with $\phi(\cdot,\Omega_2)$ gives
    \[
    I\widetilde\phi(x)+\mu\widetilde\phi(x)
       \leq(\mu-\lambda(\Omega_2))\phi(x,\Omega_2)\leq0.
    \]
    Since
    $I\eta_1+\mu\eta_1$ is bounded in $\R^n$, there exists $C>0$ such that
    \[
    I\widetilde{\phi}+\mu\widetilde{\phi}\leq C
       \quad\text{in }\R^n
    \]
    in the viscosity sense.
    Choose $\theta>0$ so small that $\mu^+\theta\leq\varepsilon$ and set $\psi=\widetilde{\phi}+\theta$.
    Then
    \[
    I\psi+\mu\psi=(I\widetilde{\phi}+\mu\widetilde{\phi})+\mu\theta\leq0
    \quad\text{in }\Omega_1,
    \]
    while the global bounded-defect estimate gives
    \[
    I\psi+\mu\psi=(I\widetilde{\phi}+\mu\widetilde{\phi})+\mu\theta
    \leq C+\mu\theta
    \quad\text{in }\R^n
    \]
    in the viscosity sense.
    Taking $C_4>0$ sufficiently large proves the result.
\end{proof}
We use directly the boundary barrier introduced in
\cite[(2.3), Lemma~3.3]{quaas2020principal}.  Fix $\gamma\in(0,\beta_1)$,
where $\beta_1$ is the constant appearing there.
Then there exists a bounded barrier function
$\xi\in C^2(\Omega)\cap C(\R^n)$, together with
constants $c_2,\delta_1>0$, such that
\[
   \xi>0\quad\text{in }\Omega,\qquad
   \xi=0\quad\text{in }\R^n\setminus\Omega,
\]
\[
   \xi=d_\Omega^\gamma\quad\text{in }\Omega_{\delta_1},
   \qquad
   \M^+\xi(x)\leq-c_2d_\Omega(x)^{\gamma-2s}
   \quad\text{in }\Omega_{\delta_1}.
\]
We now construct a supersolution of \eqref{e1}.
In the case of \eqref{K1}, we set $\Omega_0=\emptyset$ and $\lambda(\Omega_0)=\infty$.
\begin{lem}[Positive Supersolution]\label{lem-positive-supersolution}
    Suppose that either \eqref{K1} or \eqref{K2} holds, and let $\mu<\lambda(\Omega_0)$.
    Then there exists a nonnegative viscosity supersolution $\overline{w}\in C(\R^n)\cap L^\infty(\R^n)$ of
    \[
    I\overline{w}+\mu \overline{w}=k(\overline{w})^p\quad\text{in }\Omega,
    \]
    such that $\overline{w}>0$ in $\Omega$ and $\overline{w}=0$ in $\R^n\setminus\Omega$.
\end{lem}
\begin{proof}
    If \eqref{K2} holds, Lemma~\ref{lem-domain-continuity} allows us to
    choose a bounded $C^2$ domain $\Omega_1$ such that
    \[
    \overline{\Omega_0}\Subset\Omega_1,
    \qquad \mu<\lambda(\Omega_1).
    \]
    We do not require $\Omega_1\subset\Omega$; when $\overline{\Omega_0}$
    meets $\partial\Omega$, the domain $\Omega_1$ simply extends across
    $\partial\Omega$.  When \eqref{K1} holds, take a ball $\Omega_1\Subset\Omega$
    with $|\Omega_1|$ small enough for Lemma~\ref{lem-small-measure} with the present
    value of $\mu$. Then necessarily $\mu<\lambda(\Omega_1)$. Indeed,
    if $\mu\geq\lambda(\Omega_1)$ and
    $\phi_1=\phi(\cdot,\Omega_1)$, ellipticity gives
    \[
       \M^-(-\phi_1)+\mu(-\phi_1)
       \leq (\mu-\lambda(\Omega_1))(-\phi_1)\leq0
       \quad\text{in }\Omega_1,
    \]
    while $-\phi_1=0$ outside $\Omega_1$, contradicting
    Lemma~\ref{lem-small-measure}.
    Let $\psi$ be as in Lemma~\ref{lem32}, and write
    \[
    \theta:=\inf_{\R^n}\psi>0.
    \]
    Let $\xi$ be the barrier function fixed above, and let
    $\gamma$, $c_2$, and $\delta_1$ denote the associated constants.
    Decreasing $\delta_1$ if necessary, we may assume
    $\Omega\setminus\Omega_{\delta_1}\neq\emptyset$; the barrier estimate
    remains valid in the smaller boundary layer.  Since $\xi$ is positive in
    $\Omega$, set
    \[
       m_\xi:=\min_{\overline{\Omega\setminus\Omega_{\delta_1}}}\xi>0.
    \]
    Choose $\delta\in(0,1)$ sufficiently small that
    \[
       \sqrt\delta<\delta_1,\qquad
       I\xi+\mu\xi\leq0\quad\text{in }\Omega_{\sqrt\delta},
    \]
    and
    \[
       \delta^{\gamma/2}\leq
       \min\left\{m_\xi,
       \frac{\theta}{\|\psi\|_{L^\infty(\R^n)}}\right\}.
    \]
    The differential inequality follows from ellipticity and
    \cite[Lemma~3.3(b)]{quaas2020principal}: in
    $\Omega_{\sqrt\delta}$ one has $\xi=d_\Omega^\gamma$ and hence
    \[
       I\xi+\mu\xi
       \leq d_\Omega^{\gamma-2s}
       \bigl(-c_2+\mu^+d_\Omega^{2s}\bigr)\leq0
    \]
    after decreasing $\delta$ if necessary.

    Set $\Omega^\delta:=\Omega\setminus\Omega_\delta$ and
    \[
    m_\delta:=
    \begin{cases}
    1,
    &\text{if }\Omega^\delta\subset\Omega_1,\\[1mm]
    \displaystyle\min\left\{1,
      \inf_{\overline{\Omega^\delta}\setminus\Omega_1}k\right\},
    &\text{if }\Omega^\delta\setminus\Omega_1\neq\emptyset.
    \end{cases}
    \]
    In the second case, $\overline{\Omega^\delta}\setminus\Omega_1$ is compact and,
    under \eqref{K2}, disjoint from $\overline{\Omega_0}$; hence $m_\delta>0$.
    The same conclusion follows directly from \eqref{K1}.
    Set
    \[
    M_{\delta}:=\left(\frac{C_4}{\theta^p m_\delta}\right)^{1/(p-1)}
    \]
    where $C_4$ and $\psi$ are as in Lemma~\ref{lem32}.
    If $\Omega^\delta\setminus\Omega_1\neq\emptyset$, then
    \[
    I(M_\delta\psi)+\mu M_\delta\psi
    \leq C_4M_\delta
    =m_\delta(M_\delta\theta)^p
    \leq k(M_\delta\psi)^p
    \quad\text{in }\Omega^\delta\setminus\Omega_1.
    \]
    Moreover,
    \[
    I(M_\delta\psi)+\mu M_\delta\psi\leq0\leq k(M_\delta\psi)^p
    \quad\text{in }\Omega^\delta\cap\Omega_1.
    \]
    If $\Omega^\delta\subset\Omega_1$, the latter inequality holds throughout
    $\Omega^\delta$.  Thus, in either case, $M_\delta\psi$ is a viscosity
    supersolution in $\Omega^\delta$.

    Set
    \[
    L:=\frac{M_\delta\theta}{\delta^\gamma},
    \qquad
    \rho:=L\xi.
    \]
    By the choice of $\delta$ and the positive homogeneity of $I$,
    \[
    I\rho+\mu\rho\leq0\leq k\rho^p
       \quad\text{in }\Omega_{\sqrt\delta}.
    \]
    If $x\in\Omega_\delta$, then $\xi(x)=d_\Omega(x)^\gamma$ and
    \[
    \rho(x)\leq L\delta^\gamma=M_\delta\theta
       \leq M_\delta\psi(x).
    \]
    If $x\in\Omega\setminus\Omega_{\sqrt\delta}$, then
    $\xi(x)\geq\delta^{\gamma/2}$.  Indeed, if
    $d_\Omega(x)<\delta_1$, then
    $\xi(x)=d_\Omega(x)^\gamma\geq(\sqrt\delta)^\gamma$; if
    $d_\Omega(x)\geq\delta_1$, then
    $\xi(x)\geq m_\xi\geq\delta^{\gamma/2}$ by the choice of $\delta$.
    Hence
    \[
    \rho(x)\geq L\delta^{\gamma/2}
       =M_\delta\theta\delta^{-\gamma/2}
       \geq M_\delta\|\psi\|_{L^\infty(\R^n)}
       \geq M_\delta\psi(x).
    \]

    Define
    \[
    \overline w:=\min\{M_\delta\psi,\rho\}.
    \]
    We verify the pasting property explicitly.  In $\Omega_\delta$ the
    minimum equals $\rho$, which is a supersolution there.  In
    $\Omega\setminus\Omega_{\sqrt\delta}$ it equals $M_\delta\psi$, which
    is a supersolution there.  In the overlap
    $\Omega_{\sqrt\delta}\cap\Omega^\delta$, both functions are
    supersolutions.  If a lower test function touches $\overline w$ at a
    point of the overlap, choose an active branch, denoted by $w_*$.  Since
    $\overline w\leq w_*$ in $\R^n$ and
    $\overline w(x)=w_*(x)$, the patched function for $\overline w$ is
    bounded above by that for $w_*$.  Ellipticity therefore transfers the
    supersolution inequality from $w_*$ to $\overline w$.

    Finally, both $\psi$ and $\xi$ are positive in $\Omega$, while
    $\xi=0$ in $\R^n\setminus\Omega$.  Hence $\overline w>0$ in $\Omega$ and
    $\overline w=0$ in $\R^n\setminus\Omega$, as required.
\end{proof}
\subsection{Perron construction}
We record the Perron step in the precise form needed below.  Mou's general
Perron theorem \cite[Theorem~3.2]{mou2017perron} is stated under a
properness assumption in the zeroth-order variable, which need not hold for
the logistic term.  The required modification is technical rather than
structural, so the complete envelope and bump argument is deferred to
Appendix~\ref{app-perron}.
\begin{lem}[Perron construction]\label{lem-perron}
    Let $\underline w,\overline w\in C(\R^n)\cap L^\infty(\R^n)$ be,
    respectively, a positive viscosity subsolution and a positive viscosity
    supersolution of \eqref{e1}.  Assume
    $\underline w\leq\overline w$ in $\R^n$ and that both vanish in
    $\R^n\setminus\Omega$.  Define
    \[
       \mathcal F:=\bigl\{z\in USC(\R^n)\cap L^\infty(\R^n):
       \underline w\leq z\leq\overline w\text{ in }\R^n,
       \ z\text{ is a viscosity subsolution of }\eqref{e1}\bigr\}
    \]
    and
    \[
       W(x):=\sup_{z\in\mathcal F}z(x).
    \]
    Then $W$ is a continuous viscosity solution of \eqref{e1}.
\end{lem}
\subsection{Existence and nonexistence of solutions to \eqref{e1}}
\begin{proof}[Proof of Theorem~\ref{thm11}]
   We first prove existence.  Let $\underline w$ and $\overline w$ be the
   sub- and supersolutions constructed in Lemmas~\ref{lem-positive-subsolution}
   and~\ref{lem-positive-supersolution}.  Lemma~\ref{lem4} gives
   $\underline w\leq\overline w$ in $\Omega$; since both functions vanish
   in $\R^n\setminus\Omega$, the ordering holds in all of $\R^n$.
   Lemma~\ref{lem-perron} therefore
   yields a positive viscosity solution $u$ of \eqref{e1}.  Uniqueness follows
   from Lemma~\ref{lem4}.

    We now prove nonexistence outside the stated parameter range.
    Suppose first that $\mu\leq\lambda(\Omega)$ and that $u$ is a positive viscosity solution of \eqref{e1}.
    Then
    \[
    Iu+\lambda(\Omega)u
    =k u^p+\bigl(\lambda(\Omega)-\mu\bigr)u\geq0
    \quad\text{in }\Omega.
    \]
    Thus $u$ is a nonnegative viscosity subsolution of the principal
    eigenvalue equation with zero exterior values.  The second assertion of
    Proposition~\ref{prop} implies that
    $u=\alpha\phi(\cdot,\Omega)$ for some $\alpha>0$.  On the other
    hand, the eigenvalue equation gives $Iu=-\lambda(\Omega)u$, whereas
    \eqref{e1} gives $Iu=ku^p-\mu u$.  Both right-hand sides are
    continuous.  Uniqueness of the continuous forcing therefore yields
    \[
       -\lambda(\Omega)u=ku^p-\mu u\quad\text{in }\Omega.
    \]
    At any point of the nonempty set $\{k>0\}$, the equivalent identity
    $ku^p+(\lambda(\Omega)-\mu)u=0$ has a strictly positive left-hand
    side, a contradiction.

    Finally, assume that \eqref{K2} holds and that $\mu\geq\lambda(\Omega_0)$.
    Since $k=0$ in $\Omega_0$, any positive solution would satisfy
    \[
    Iu+\lambda(\Omega_0)u
    =\bigl(\lambda(\Omega_0)-\mu\bigr)u\leq0
    \quad\text{in }\Omega_0.
    \]
    Thus $u$ is a nonnegative viscosity supersolution of the principal
    eigenvalue equation in $\Omega_0$.  The first assertion of
    Proposition~\ref{prop}, applied in $\Omega_0$, gives
    $u=\alpha\phi(\cdot,\Omega_0)$ in $\Omega_0$ and almost everywhere in
    $\R^n$ for some $\alpha>0$.  The almost-everywhere identity would then
    give $u=0$ almost everywhere on the nonempty open set
    $\{k>0\}=\Omega\setminus\overline{\Omega_0}$.  This contradicts the
    continuity and strict positivity of $u$ there.
\end{proof}
\section{Asymptotic analysis}
\subsection{An auxiliary local estimate}
For the local estimate below, define the weight
\[
\omega_s(y):=\frac{1}{1+|y|^{n+2s}},\qquad y\in\R^n,
\]
and the corresponding weighted tail norm
\[
\|w\|_{L^1_{\omega_s}(\R^n)}
:=\int_{\R^n}|w(y)|\,\omega_s(y)\,dy.
\]
We use the following local maximum principle of Caffarelli--Silvestre
\cite[Theorem~5.1]{caffarelli2011evans}; see also Fern\'andez-Real--Ros-Oton
\cite[Theorem~3.3.5]{fernandezreal2024integro}.
\begin{lem}[Local maximum principle]\label{lem28}
    There exists a constant $C_3=C_3(n,s,\lambda,\Lambda)>0$ with the
    following property.  Let $x_0\in\R^n$ and $r>0$.  If
    $u\in C(\overline{B_{4r}(x_0)})$ has finite
    $\|u\|_{L^1_{\omega_s}(\R^n)}$ and is a viscosity subsolution of
    \[
    \M^+u=f\quad\text{in }B_{4r}(x_0),
    \]
    where $f\in C(\overline{B_{4r}(x_0)})$, then
    \[
    \sup_{B_{2r}(x_0)}u\leq C_3\left(
       \|u(x_0+r\,\cdot)\|_{L^1_{\omega_s}(\R^n)}
       +r^{2s}\|f^-\|_{L^\infty(B_{4r}(x_0))}\right).
    \]
\end{lem}

\subsection{Endpoint asymptotics}
Throughout this subsection, $C_{n,s}$ denotes a positive constant depending
only on $n$ and $s$, whose value may change from line to line.
\begin{proof}[Proof of Theorem~\ref{thm-lower-asymptotics}]
    We begin with monotonicity in the parameter, which will also be used
    below.  If
    $\lambda(\Omega)<\mu_1<\mu_2<\lambda(\Omega_0)$, then
    \[
       Iu_{\mu_1}+\mu_2u_{\mu_1}
       =ku_{\mu_1}^p+(\mu_2-\mu_1)u_{\mu_1}
       \geq ku_{\mu_1}^p.
    \]
    Hence $u_{\mu_1}$ is a viscosity subsolution for the problem with parameter
    $\mu_2$, and Lemma~\ref{lem4} gives
    $u_{\mu_1}\leq u_{\mu_2}$.

    Fix $\bar\mu$ in the existence interval with
    $\bar\mu>\lambda(\Omega)$.  By the preceding monotonicity,
    $0<u_\mu\leq u_{\bar\mu}$ whenever
    $\lambda(\Omega)<\mu\leq\bar\mu$.  Thus the family is uniformly
    bounded as $\mu\downarrow\lambda(\Omega)$.  Since
    $Iu_\mu=ku_\mu^p-\mu u_\mu$, the right-hand sides are uniformly
    bounded, and Lemma~\ref{lem3} gives a uniform global H\"older bound.
    Hence every sequence $\mu_j\downarrow\lambda(\Omega)$ has a
    subsequence, not relabeled, for which $u_{\mu_j}$ converges uniformly in
    $\R^n$ to some $u_0\in C(\R^n)$.  By the stability of viscosity solutions
    \cite[Corollary~4.7]{caffarelli2009regularity}, the limit satisfies,
    in the viscosity sense,
   \begin{equation}\label{Geq1}
    Iu_0+\lambda(\Omega)u_0=k u_0^p\quad\text{in }\Omega.
    \end{equation}
    In particular, $Iu_0+\lambda(\Omega)u_0\geq0$, so $u_0$ is a
    nonnegative viscosity subsolution of the principal-eigenvalue equation,
    with zero exterior values.  The second assertion of
    Proposition~\ref{prop} therefore gives
    $u_0=\alpha\phi(\cdot,\Omega)$ for some $\alpha\geq0$.
    Using \eqref{Geq1} and the eigenvalue equation, we obtain
    \[
    k u_0^p=0\quad\text{in }\Omega.
    \]
    Since $u_0=\alpha\phi(\cdot,\Omega)$, $k\not\equiv0$, and
    $\phi>0$ in $\Omega$, it follows that $\alpha=0$.  Since the initial sequence was
    arbitrary, the full family converges uniformly to zero.

    Set $M_\mu:=\|u_\mu\|_{L^\infty(\Omega)}$ and $v_\mu:=u_\mu/M_\mu$.
    Then $M_\mu\to0$ and, in the viscosity sense,
    \[
    Iv_\mu+\mu v_\mu=kM_\mu^{p-1}v_\mu^p\quad\text{in }\Omega.
    \]
    By the same compactness argument, every sequence
    $\mu_j\downarrow\lambda(\Omega)$ has a subsequence along which
    $v_{\mu_j}$ converges uniformly in $\R^n$ to a nonnegative function
    $v$.  By the stability of viscosity solutions
    \cite[Corollary~4.7]{caffarelli2009regularity}, the normalized
    equations imply that $v$ is a viscosity solution of
    \[
    Iv+\lambda(\Omega)v=0\quad\text{in }\Omega,
    \qquad v=0\quad\text{in }\R^n\setminus\Omega.
    \]
    By Proposition~\ref{prop}, the limit is a multiple of
    $\phi(\cdot,\Omega)$.  Uniform convergence preserves the
    $L^\infty(\R^n)$-norm, and both $v_\mu$ and
    $\phi(\cdot,\Omega)$ have norm one.  Hence the multiple equals one.
    Since every sequence admits only this possible subsequential limit, the
    asserted convergence holds for the full family.
\end{proof}

\begin{proof}[Proof of the first two assertions of Theorem~\ref{thm-upper-asymptotics}]
    Write
    \[
    M_\mu:=\|u_\mu\|_{L^\infty(\Omega)},
    \qquad v_\mu:=\frac{u_\mu}{M_\mu}.
    \]
    We first show that $M_\mu\to\infty$ as $\mu\uparrow\lambda(\Omega_0)$.
    If not, there would exist a sequence
    $\mu_j\uparrow\lambda(\Omega_0)$ for which $M_{\mu_j}$ remains bounded.
    The equation and Lemma~\ref{lem3} would then give a uniform global
    H\"older bound, so after passing to a subsequence
    $u_{\mu_j}\to u_*$ uniformly in $\R^n$.  By the stability of viscosity solutions
    \cite[Corollary~4.7]{caffarelli2009regularity}, $u_*$ is a viscosity
    solution of \eqref{e1} with $\mu=\lambda(\Omega_0)$.  Fix
    $\mu_*\in(\lambda(\Omega),\lambda(\Omega_0))$.  For all sufficiently
    large $j$, the monotonicity established above gives
    $u_{\mu_j}\geq u_{\mu_*}>0$ in $\Omega$.  Hence $u_*>0$ in $\Omega$,
    contradicting Theorem~\ref{thm11}.

    Since $k=0$ in $\Omega_0$, we have, in the viscosity sense,
    \[
    Iv_\mu+\mu v_\mu=0\quad\text{in }\Omega_0.
    \]
    Interior H\"older regularity
    \cite[Theorem~12.1]{caffarelli2009regularity} gives local compactness in
    $C(\Omega_0)$.
    Thus, along any sequence $\mu_j\uparrow\lambda(\Omega_0)$, write
    $M_j:=M_{\mu_j}$ and $v_j:=v_{\mu_j}$.  After passing to a subsequence,
    we may assume that
    \[
       v_j\to v_\infty
       \quad\text{locally uniformly in }\Omega_0.
    \]

    Before passing to the limit in the equation, we identify the behavior
    outside $\Omega_0$.  We claim that
    \begin{equation}\label{normalized-vanishing-away-from-sanctuary}
    v_j\to0
    \quad\text{locally uniformly in }
       \Omega\setminus\overline{\Omega_0}.
    \end{equation}
    Fix $K'\Subset\Omega\setminus\overline{\Omega_0}$.  Choose $r_K>0$
    and $\kappa_K>0$ so that
    \[
    \{x\in\mathbb R^n:\operatorname{dist}(x,K')\leq 2r_K\}
       \Subset\Omega\setminus\overline{\Omega_0},
    \]
    and
    \[
    k\geq\kappa_K
    \quad\text{on }
    \{x\in\mathbb R^n:\operatorname{dist}(x,K')\leq 2r_K\}.
    \]
    Set $A_K:=4/r_K^2$.  For $x\in K'$, choose
    $y_{j,x}\in\overline{B_{r_K}(x)}$ so that
    \[
       v_j(y_{j,x})-A_K|x-y_{j,x}|^2
       =\max_{y\in\overline{B_{r_K}(x)}}
       \bigl\{v_j(y)-A_K|x-y|^2\bigr\}.
    \]
    Comparison with $y=x$, together with $0\leq v_j\leq1$, gives
    \[
    |x-y_{j,x}|\leq\frac{r_K}{2},
    \qquad v_j(x)\leq v_j(y_{j,x}).
    \]
    With $\rho_K:=r_K/4$, the quadratic polynomial
    \[
    \psi_{j,x}(y):=v_j(y_{j,x})-A_K|x-y_{j,x}|^2
       +A_K|x-y|^2
    \]
    touches $v_j$ from above at $y_{j,x}$ in $B_{\rho_K}(y_{j,x})$.
    Let $Z_{j,x}$ be the function obtained by replacing $v_j$ by
    $\psi_{j,x}$ in this ball.  At $y_{j,x}$,
    \[
    \delta(Z_{j,x},y_{j,x},h)=2A_K|h|^2
       \quad\text{if }|h|<\rho_K,
    \]
    while, if $|h|\geq\rho_K$,
    \[
    \delta(Z_{j,x},y_{j,x},h)
       =\delta(v_j,y_{j,x},h)\leq2.
    \]
    Hence, uniformly in $a,b,j$, and $x\in K'$,
    \begin{align*}
    L_{a,b}Z_{j,x}(y_{j,x})
    &\leq2A_K\Lambda\int_{B_{\rho_K}}|h|^{2-n-2s}\,dh
       +2\Lambda\int_{\R^n\setminus B_{\rho_K}}
          |h|^{-n-2s}\,dh\\
    &\leq C_{n,s}\Lambda r_K^{-2s}.
    \end{align*}
    The normalized equation and the viscosity subsolution inequality now
    imply
    \begin{align*}
    \kappa_KM_j^{p-1}v_j(x)^p
    &\leq k(y_{j,x})M_j^{p-1}v_j(y_{j,x})^p\\
    &\leq IZ_{j,x}(y_{j,x})+\mu_jv_j(y_{j,x})\\
    &\leq C_{n,s}\Lambda r_K^{-2s}+\sup_j|\mu_j|.
    \end{align*}
    Since $M_j\to\infty$, this proves
    \eqref{normalized-vanishing-away-from-sanctuary}.

    Extend $v_\infty$ by zero in $\R^n\setminus\Omega_0$.  The boundaries
    of the $C^2$ domains $\Omega$ and $\Omega_0$ have measure zero.
    Therefore the local convergence in $\Omega_0$,
    \eqref{normalized-vanishing-away-from-sanctuary}, and the exterior
    condition imply
    \[
       v_j\to v_\infty\quad\text{almost everywhere in }\R^n.
    \]
    Since
    \[
       Iv_j=-\mu_jv_j\quad\text{in }\Omega_0,
    \]
    the local uniform convergence in $\Omega_0$, the almost-everywhere
    convergence in $\R^n$, and the stability of viscosity solutions
    \cite[Corollary~4.7]{caffarelli2009regularity} give, in the viscosity
    sense,
    \[
       Iv_\infty+\lambda(\Omega_0)v_\infty=0
       \quad\text{in }\Omega_0.
    \]
    Proposition~\ref{prop} then yields
    \[
    v_\infty=\alpha\phi(\cdot,\Omega_0)
    \quad\text{in }\Omega_0\text{ and almost everywhere in }\R^n
    \]
    for some $\alpha\geq0$.

    We next upgrade this convergence to uniform convergence in $\R^n$.
    Fix $\varepsilon\in(0,1)$.
    For $\eta>0$, we first note that
    \begin{equation}\label{scaled-tail-eigenfunction}
       \sup_{x\in\overline\Omega\setminus\Omega_0}
       \|\alpha\phi(x+\eta\,\cdot,\Omega_0)\|_{L^1_{\omega_s}(\R^n)}
       \longrightarrow0\qquad\text{as }\eta\downarrow0.
    \end{equation}
    Indeed, since $\phi_0:=\phi(\cdot,\Omega_0)$ is bounded and uniformly
    continuous in $\R^n$ and vanishes on $\R^n\setminus\Omega_0$,
    \eqref{scaled-tail-eigenfunction} follows by dominated convergence.
    Hence, using also the continuity of
    $\phi_0$ up to $\partial\Omega_0$ and the boundedness of
    $\{\mu_j\}$, we may choose $\eta>0$ so small that
    \[
    \|\alpha\phi(\cdot,\Omega_0)\|_{L^\infty((\Omega_0)_\eta)}
    \leq\frac{\varepsilon}{3}
    \]
    and, uniformly for $x\in\overline{\Omega}\setminus\Omega_0$,
    \[
    C_3\left(
      \|\alpha\phi(x+\eta\,\cdot,\Omega_0)\|_{L^1_{\omega_s}(\R^n)}
      +\eta^{2s}\sup_j|\mu_j|
    \right)\leq\frac{\varepsilon}{3}.
    \]
    Local uniform convergence then gives, for all sufficiently large $j$,
    \[
    \|v_j-\alpha\phi(\cdot,\Omega_0)\|_{L^\infty(\Omega_0\setminus(\Omega_0)_\eta)}
    \leq\frac{\varepsilon}{3}.
    \]
    By compactness, choose finitely many points
    $x_1,\dots,x_N\in\overline{\Omega}\setminus\Omega_0$ such that the
    balls $B_{2\eta}(x_i)$ cover
    \[
    (\overline{\Omega}\setminus\Omega_0)\cup\overline{(\Omega_0)_\eta}.
    \]
    By ellipticity, $v_j$ satisfies, in the viscosity
    sense,
    \[
    \M^+v_j\geq-\mu_jv_j\quad\text{in }\Omega.
    \]
    The local maximum principle may also be applied to balls meeting
    $\partial\Omega$.
    Lemma~\ref{lem28} therefore gives
    \begin{equation}\label{e422}
    \sup_{B_{2\eta}(x_i)}v_j
    \leq C_3\left(
       \|v_j(x_i+\eta\,\cdot)\|_{L^1_{\omega_s}(\R^n)}
       +\eta^{2s}|\mu_j|
    \right).
    \end{equation}
    The almost-everywhere convergence above and dominated convergence imply,
    for fixed $\eta$ and $i$, as $j\to\infty$,
    \[
    \|v_j(x_i+\eta\,\cdot)
       -\alpha\phi(x_i+\eta\,\cdot,\Omega_0)\|_{L^1_{\omega_s}(\R^n)}
       \longrightarrow0.
    \]
    By the choice of $\eta$ and the preceding convergence, the right-hand
    side of \eqref{e422} is at most $2\varepsilon/3$ for every $i$ and all
    sufficiently large $j$.  Together with the choice of $\eta$ and the
    interior estimate on $\Omega_0\setminus(\Omega_0)_\eta$, this proves
    \[
       \|v_j-\alpha\phi(\cdot,\Omega_0)\|_{L^\infty(\R^n)}
       \leq\varepsilon.
    \]
    Since $\|v_j\|_{L^\infty(\R^n)}=1$, it follows that $\alpha=1$.
    By the uniqueness of the limiting function, the full family satisfies
    \begin{equation}\label{e421}
       v_\mu\to\phi(\cdot,\Omega_0)
       \quad\text{uniformly in }\R^n
    \end{equation}
    as $\mu\uparrow\lambda(\Omega_0)$.

    Since $M_\mu\to\infty$ and $\phi(\cdot,\Omega_0)>0$ in $\Omega_0$, this already proves that $u_\mu\to\infty$ locally uniformly in $\Omega_0$.
    We now propagate the divergence to all of $\Omega$.
    Suppose, to the contrary, that there exist a compact set
    $\widetilde K\Subset\Omega$, a sequence
    $\mu_j\uparrow\lambda(\Omega_0)$, points
    $x_j\in\widetilde K$, and a constant $L>0$ such that, writing
    $u_j:=u_{\mu_j}$,
    \[
       u_j(x_j)\leq L\qquad\text{for every }j.
    \]
    After passing to a subsequence, assume that $x_j\to\widehat x\in\widetilde K$.
    Choose a closed ball $E\Subset\Omega_0$ with $\widehat x\notin E$ and
    $r>0$ so small that
    \[
    B_r(x_j)\subset B_{2r}(\widehat x)\Subset\Omega,
    \qquad
    \dist\bigl(E,B_{2r}(\widehat x)\bigr)>r
    \]
    for all sufficiently large $j$.  Set
    \[
    A:=\frac{4L}{r^2}.
    \]
    Let $y_j\in\overline{B_r(x_j)}$ minimize
    \[
    y\longmapsto u_j(y)+A|x_j-y|^2.
    \]
    Positivity of $u_j$ and comparison with the value at $x_j$ show that
    \[
    |x_j-y_j|\leq\frac r2,
    \qquad u_j(y_j)\leq L.
    \]
    Put $\rho:=r/4$.  Then $B_\rho(y_j)\subset B_r(x_j)$, and
    \[
    \varphi_j(y):=u_j(y_j)+A|x_j-y_j|^2-A|x_j-y|^2
    \]
    touches $u_j$ from below at $y_j$ in $B_\rho(y_j)$.  Accordingly, define
    \[
    z_j(y):=
    \begin{cases}
    \varphi_j(y),&y\in B_\rho(y_j),\\[1mm]
    u_j(y),&y\in\R^n\setminus B_\rho(y_j).
    \end{cases}
    \]
    The viscosity supersolution inequality gives
    \[
    Iz_j(y_j)+\mu_j u_j(y_j)\leq k(y_j)u_j(y_j)^p.
    \]
    To estimate the nonlocal term, set $m_j:=\inf_E u_j$.  Since
    $E\Subset\Omega_0$, we know that $m_j\to\infty$.  For
    $|h|<\rho$,
    \begin{equation}\label{local-quadratic-difference}
    \delta(z_j,y_j,h)=-2A|h|^2.
    \end{equation}
    If $|h|\geq\rho$, then $z_j(y_j\pm h)=u_j(y_j\pm h)$.  Hence,
    since $u_j\geq0$ and $u_j(y_j)\leq L$,
    \begin{equation}\label{far-difference-lower-bound}
    \delta(z_j,y_j,h)
    =u_j(y_j+h)+u_j(y_j-h)-2u_j(y_j)
    \geq-2L.
    \end{equation}
    Finally, for $h\in E-y_j$, the separation of $E$ from
    $B_{2r}(\widehat x)$ and the nonnegativity of $u_j$ yield
    \begin{equation}\label{sanctuary-difference-lower-bound}
    \delta(z_j,y_j,h)
    =u_j(y_j+h)+u_j(y_j-h)-2u_j(y_j)
    \geq m_j-2L.
    \end{equation}
    For each $a\in A$ and $b\in B$, write
    \[
    L_{a,b}w(x):=\int_{\R^n}\delta(w,x,h)K_{a,b}(h)\,dh.
    \]
    The separation chosen above implies
    $E-y_j\subset\R^n\setminus B_\rho$.  Hence $B_\rho$, $E-y_j$, and
    $\R^n\setminus\bigl(B_\rho\cup(E-y_j)\bigr)$ form a disjoint partition of
    $\R^n$, up to their boundaries.  We use
    \eqref{local-quadratic-difference} on $B_\rho$,
    \eqref{sanctuary-difference-lower-bound} on $E-y_j$, and
    \eqref{far-difference-lower-bound} on
    $\R^n\setminus\bigl(B_\rho\cup(E-y_j)\bigr)$.  The bounds on the kernels
    then give, for all sufficiently large indices $j$,
    \begin{align*}
    L_{a,b}z_j(y_j)
    &\geq \lambda(m_j-2L)
       \int_{E-y_j}\frac{dh}{|h|^{n+2s}}
       -2A\Lambda\int_{B_\rho}\frac{|h|^2}{|h|^{n+2s}}\,dh\\
    &\hspace{25mm}
       -2L\Lambda\int_{\R^n\setminus\left(B_\rho\cup(E-y_j)\right)}
          \frac{dh}{|h|^{n+2s}}\\
    &\geq \lambda(m_j-2L)
       \int_{E-y_j}\frac{dh}{|h|^{n+2s}}
       -2A\Lambda\int_{B_\rho}|h|^{2-n-2s}\,dh\\
    &\hspace{25mm}
       -2L\Lambda\int_{\R^n\setminus B_\rho}
          |h|^{-n-2s}\,dh\\
    &\geq
       \frac{\lambda|E|}{(\operatorname{diam}\Omega)^{n+2s}}(m_j-2L)
       -C_{n,s}\Lambda Lr^{-2s}.
    \end{align*}
    In the last line we used $A=4L/r^2$, $\rho=r/4$, and the elementary radial estimates
    \[
    \int_{B_\rho}|h|^{2-n-2s}\,dh
       \leq C_{n,s}\rho^{2-2s},
    \qquad
    \int_{\R^n\setminus B_\rho}|h|^{-n-2s}\,dh
       \leq C_{n,s}\rho^{-2s}.
    \]
    The bound is uniform in $a$ and $b$.  Therefore
    \begin{equation}
    Iz_j(y_j)=\inf_{a\in A}\sup_{b\in B}L_{a,b}z_j(y_j)
    \geq
       \frac{\lambda|E|}{(\operatorname{diam}\Omega)^{n+2s}}(m_j-2L)
       -C_{n,s}\Lambda Lr^{-2s},
    \end{equation}
    whose right-hand side tends to $+\infty$.  On the other hand, the
    viscosity inequality gives
    \[
    Iz_j(y_j)\leq
    \|k\|_{L^\infty(\Omega)}L^p
       +\left(\sup_j|\mu_j|\right)L,
    \]
    a contradiction.  This proves the locally uniform divergence throughout
    $\Omega$.
\end{proof}

\begin{proof}[Proof of the final assertion of Theorem~\ref{thm-upper-asymptotics}]
    Write
    \[
       M_\mu:=\|u_\mu\|_{L^\infty(\Omega)},
       \qquad
       q_\mu:=\frac{u_\mu}{M_\mu^{1/p}},
       \qquad
       \phi_0:=\phi(\cdot,\Omega_0).
    \]
    Fix a compact set $K\Subset\Omega\setminus\overline{\Omega_0}$ and
    choose $r>0$ so that
    \begin{equation}\label{d}
       \dist(K,\Omega_0)>2r,
       \qquad
       \dist(K,\R^n\setminus\Omega)>2r.
    \end{equation}
    We first verify that the proposed limit is well defined and continuous on
    $K$.  If $x\in K$, then $\phi_0$ vanishes in a neighborhood of $x$, so
    $I\phi_0(x)$ is classically defined.  By the symmetry of the kernels
    and (H1), $I\phi_0(x)>0$.
    By the uniform continuity of $\phi_0$, \eqref{d}, and the kernel bounds,
    \[
       \sup_{a,b}|L_{a,b}\phi_0(x)-L_{a,b}\phi_0(x')|\longrightarrow0
       \qquad
       \text{as }x'\to x\text{ in }K.
    \]
    The elementary estimate
    \[
       \left|\inf_a\sup_b A_{a,b}-\inf_a\sup_b B_{a,b}\right|
       \leq\sup_{a,b}|A_{a,b}-B_{a,b}|
    \]
    shows that $I\phi_0$ is continuous on $K$.  Since $k$ is continuous and
    strictly positive there, the function
    \[
       G(x):=\frac{I\phi_0(x)}{k(x)}
    \]
    is positive and continuous on $K$.

    We prove
    \begin{equation}\label{eq-second-scale-power}
       \|q_\mu^p-G\|_{L^\infty(K)}\longrightarrow0
       \qquad\text{as }\mu\uparrow\lambda(\Omega_0).
    \end{equation}
    Suppose, to the contrary, that \eqref{eq-second-scale-power} fails.
    Then there exist $\varepsilon_0>0$, a sequence
    $\mu_j\uparrow\lambda(\Omega_0)$, and points $x_j\in K$ such that,
    writing
    \[
       u_j:=u_{\mu_j},\qquad
       M_j:=M_{\mu_j},\qquad
       v_j:=\frac{u_j}{M_j},\qquad
       q_j:=\frac{u_j}{M_j^{1/p}},
    \]
    one has
    \begin{equation}\label{eq-second-scale-failure}
       |q_j(x_j)^p-G(x_j)|\geq\varepsilon_0
       \qquad\text{for every }j.
    \end{equation}
    After passing to a subsequence, assume $x_j\to x_\infty\in K$.
    We first prove the upper bound.  Let
    $y_j\in\overline{B_r(x_j)}$ maximize
    \[
       y\longmapsto q_j(y)-M_j^{1-1/p}|x_j-y|^2.
    \]
    Since $q_j=M_j^{1-1/p}v_j$, comparison with $y=x_j$ gives
    \[
       |x_j-y_j|^2\leq v_j(y_j)-v_j(x_j),
       \qquad
       q_j(x_j)\leq q_j(y_j).
    \]
    Since $v_j\to\phi_0$ uniformly in $\R^n$ and
    $\dist(K,\Omega_0)>2r$, the right-hand side tends to zero; hence
    $|x_j-y_j|\to0$.

    Fix $\eta\in(0,r/2)$.  For all sufficiently large $j$, define
    \[
       z_j^\eta(y):=
       \begin{cases}
       |x_j-y|^2+v_j(y_j)-|x_j-y_j|^2,
          &y\in B_\eta(y_j),\\
       v_j(y),&y\in\R^n\setminus B_\eta(y_j).
       \end{cases}
    \]
    The maximizing property of $y_j$ shows that $z_j^\eta$ is the patched
    upper test for $v_j$ at $y_j$.  At the touching point,
    \begin{equation}\label{upper-patched-difference}
       \delta(z_j^\eta,y_j,h)=
       \begin{cases}
       2|h|^2,&|h|<\eta,\\
       \delta(v_j,y_j,h),&|h|\geq\eta.
       \end{cases}
    \end{equation}
    The viscosity subsolution inequality for
    $Iv_j+\mu_jv_j=kq_j^p$ gives
    \[
       Iz_j^\eta(y_j)+\mu_jv_j(y_j)
       \geq k(y_j)q_j(y_j)^p.
    \]
    Therefore
    \begin{equation}\label{eq-upper-qj}
       \limsup_{j\to\infty}q_j(x_j)^p
       \leq
       \limsup_{j\to\infty}
       \frac{Iz_j^\eta(y_j)+\mu_jv_j(y_j)}{k(y_j)}.
    \end{equation}

    Since $v_j\to\phi_0$ uniformly in $\R^n$, $y_j\to x_\infty$,
    and $\phi_0$ is uniformly continuous,
    \[
       \sup_{h\in\R^n}
       |\delta(v_j,y_j,h)-\delta(\phi_0,x_\infty,h)|\longrightarrow0.
    \]
    Since $\delta(\phi_0,x_\infty,h)=0$ for $|h|<2r$,
    \eqref{upper-patched-difference} and the kernel bounds give, for every
    fixed $\eta\in(0,r/2)$,
    \begin{equation}\label{upper-operator-limit}
       \limsup_{j\to\infty}
       |Iz_j^\eta(y_j)-I\phi_0(x_\infty)|
       \leq C_{n,s}\Lambda\eta^{2-2s}.
    \end{equation}
    From \eqref{eq-upper-qj}--\eqref{upper-operator-limit},
    \[
       \limsup_{j\to\infty}q_j(x_j)^p
       \leq\frac{I\phi_0(x_\infty)
          +C_{n,s}\Lambda\eta^{2-2s}}{k(x_\infty)}.
    \]
    Letting the patch radius $\eta\downarrow0$ yields
    \begin{equation}\label{eq-upper-second-scale}
       \limsup_{j\to\infty}q_j(x_j)^p\leq G(x_\infty).
    \end{equation}

    For the lower bound, let
    $\widehat y_j\in\overline{B_r(x_j)}$ minimize
    \[
       y\longmapsto q_j(y)+M_j^{1-1/p}|x_j-y|^2.
    \]
    Repeating the preceding argument with lower tests and the viscosity
    supersolution inequality gives
    \begin{equation}\label{eq-lower-second-scale}
       \liminf_{j\to\infty}q_j(x_j)^p\geq G(x_\infty).
    \end{equation}

    Combining \eqref{eq-upper-second-scale} and
    \eqref{eq-lower-second-scale}, we obtain
    $q_j(x_j)^p\to G(x_\infty)$.  Since $G$ is continuous on $K$ and
    $x_j\to x_\infty$, also $G(x_j)\to G(x_\infty)$, contradicting
    \eqref{eq-second-scale-failure}.  This proves
    \eqref{eq-second-scale-power}.

    Finally, the map $t\mapsto t^{1/p}$ is uniformly continuous on every
    bounded interval of $[0,\infty)$.  Since $G$ is bounded on $K$,
    \eqref{eq-second-scale-power} implies
    \[
       \left\|\frac{u_\mu}{M_\mu^{1/p}}-G^{1/p}\right\|_{L^\infty(K)}
       \longrightarrow0
       \qquad\text{as }\mu\uparrow\lambda(\Omega_0).
    \]
    As $K\Subset\Omega\setminus\overline{\Omega_0}$ was arbitrary, this is
    the asserted locally uniform convergence.
\end{proof}
\appendix
\section{Proof of the difference inequality}\label{app-difference}

Lemma~2.2 is a slight modification of
\cite[Lemma~5.8]{caffarelli2009regularity}; the additional point is the
state-dependent zeroth-order term. We also use the standard sup/inf-convolution
argument; see \cite[Section~5]{caffarelli2009regularity} and
\cite[Section~3]{crandall1992users}.

\begin{lem}[Patched-test stability]\label{lem-patched-stability}
Let $x_j\to x$ and let $q_j,q\in L^\infty(\R^n)$ be uniformly bounded.
Assume that $q_j\to q$ in $C^2$ on a fixed neighborhood of $x$ and
\[
   \int_{\R^n}\frac{|q_j(y)-q(y)|}{1+|y|^{n+2s}}\,dy\longrightarrow0.
\]
Then
\[
   Iq_j(x_j)\to Iq(x),
   \qquad
   \M^\pm q_j(x_j)\to \M^\pm q(x).
\]
\end{lem}

\begin{proof}
Fix $0<\rho<R$.  On $|h|<\rho$, the conclusion follows from the local
$C^2$ convergence.  On $\rho\leq |h|\leq R$, the weighted convergence and
translation continuity in $L^1_{\rm loc}$ control the moving translates, while
on $|h|>R$ the contribution is uniformly small by boundedness.  Letting first
$j\to\infty$, then $\rho\downarrow0$ and $R\uparrow\infty$, gives the claim.
The kernel bounds make these estimates uniform in $a,b$, so the infimum and
supremum may be passed to the limit.
\end{proof}

\begin{proof}[Proof of Lemma~\ref{lem-difference}]
    Let $\psi$ touch $w=v-u$ from below at $x_0\in\Omega$. For $\delta>0$, replace $\psi$ by
    \[
       \psi_\delta(x):=\psi(x)-\delta|x-x_0|^4,
    \]
    so that the contact is strict. We carry out the argument for $\psi_\delta$ and let $\delta\downarrow0$ at the end; the passage to the limit is justified by Lemma~\ref{lem-patched-stability}. To simplify notation, we continue to write $\psi$.
    Shrinking the contact neighborhood, choose $r>0$ so
    that $\psi\in C^2(\overline{B_{4r}(x_0)})$ and
    $\overline{B_{4r}(x_0)}\Subset\Omega$.  For $x\in B_{2r}(x_0)$ set
    \[
       u^\varepsilon(x)
       :=\sup_{y\in\overline{B_{3r}(x_0)}}
       \left\{u(y)-\frac{|x-y|^2}{2\varepsilon}\right\},
       \qquad
       v_\varepsilon(x)
       :=\inf_{z\in\overline{B_{3r}(x_0)}}
       \left\{v(z)+\frac{|x-z|^2}{2\varepsilon}\right\},
    \]
    and $w_\varepsilon:=v_\varepsilon-u^\varepsilon$.  The extrema are
    attained.  By a standard sup/inf-convolution argument, there are
    $x_\varepsilon\to x_0$ and $c_\varepsilon\to0$ such that
    $\psi+c_\varepsilon$ touches $w_\varepsilon$ from below at
    $x_\varepsilon\in B_r(x_0)$ and
    $w_\varepsilon(x_\varepsilon)\to w(x_0)$.

    Choose active points $y_\varepsilon,z_\varepsilon$ at $x_\varepsilon$ and
    write
    \[
       h_\varepsilon:=y_\varepsilon-x_\varepsilon,
       \qquad
       \ell_\varepsilon:=z_\varepsilon-x_\varepsilon,
       \qquad
       \alpha_\varepsilon:=\frac{|h_\varepsilon|^2}{2\varepsilon},
       \qquad
       \beta_\varepsilon:=\frac{|\ell_\varepsilon|^2}{2\varepsilon}.
    \]
    Boundedness gives
    $|h_\varepsilon|+|\ell_\varepsilon|\leq C\sqrt\varepsilon$; hence, for
    small $\varepsilon$, the active points lie in $B_{2r}(x_0)\Subset\Omega$.
    Moreover
    \[
       w_\varepsilon(x_\varepsilon)
       =v(z_\varepsilon)-u(y_\varepsilon)
        +\alpha_\varepsilon+\beta_\varepsilon.
    \]
    By semicontinuity and the convergence
    $w_\varepsilon(x_\varepsilon)\to w(x_0)$, we have
    \[
       v(z_\varepsilon)-u(y_\varepsilon)\longrightarrow w(x_0),
       \qquad
       \alpha_\varepsilon+\beta_\varepsilon\longrightarrow0.
    \]

    We next transfer the viscosity inequalities at the active points.  The
    functions $u^\varepsilon$ and $v_\varepsilon$ are respectively semiconvex
    and semiconcave.  Since $\psi+c_\varepsilon$ touches
    $v_\varepsilon-u^\varepsilon$ from below at $x_\varepsilon$, the argument
    of \cite[Lemma~5.8]{caffarelli2009regularity} shows that $u^\varepsilon$ and
    $v_\varepsilon$ are both punctually $C^{1,1}$ at $x_\varepsilon$.  Hence the translated functions
    \[
       X\longmapsto u^\varepsilon(X-h_\varepsilon)+\alpha_\varepsilon,
       \qquad
       X\longmapsto v_\varepsilon(X-\ell_\varepsilon)-\beta_\varepsilon
    \]
    are punctually $C^{1,1}$ at $y_\varepsilon$ and $z_\varepsilon$,
    respectively, and may be used as viscosity test functions by
    \cite[Lemma~4.3]{caffarelli2009regularity}.  For small $\varepsilon$ the
    first lies above $u$ on $B_r(x_0+h_\varepsilon)$ and equals $u$ at
    $y_\varepsilon$, while the second lies below $v$ on
    $B_r(x_0+\ell_\varepsilon)$ and equals $v$ at $z_\varepsilon$.  After
    patching with $u,v$ outside these balls and
    translating back to $x_\varepsilon$, set
    \[
       U_\varepsilon(X):=
       \begin{cases}
          u^\varepsilon(X),&X\in B_r(x_0),\\
          u(X+h_\varepsilon)-\alpha_\varepsilon,
             &X\in\R^n\setminus B_r(x_0),
       \end{cases}
    \]
    and
    \[
       V_\varepsilon(X):=
       \begin{cases}
          v_\varepsilon(X),&X\in B_r(x_0),\\
          v(X+\ell_\varepsilon)+\beta_\varepsilon,
             &X\in\R^n\setminus B_r(x_0).
       \end{cases}
    \]
    Translation and constant invariance give
    \[
       IU_\varepsilon(x_\varepsilon)
       \geq F(y_\varepsilon,u(y_\varepsilon)),
       \qquad
       IV_\varepsilon(x_\varepsilon)
       \leq G(z_\varepsilon,v(z_\varepsilon)).
    \]

    Let $q_\varepsilon$ have local branch $\psi+c_\varepsilon$ in
    $B_r(x_0)$ and exterior branch $V_\varepsilon-U_\varepsilon$.  Since
    $\psi+c_\varepsilon\leq w_\varepsilon
    =V_\varepsilon-U_\varepsilon$ throughout $B_r(x_0)$, with equality at
    $x_\varepsilon$, ellipticity gives
    \begin{align*}
       \M^-q_\varepsilon(x_\varepsilon)
       &\leq IV_\varepsilon(x_\varepsilon)
              -IU_\varepsilon(x_\varepsilon)\\
       &\leq G(z_\varepsilon,v(z_\varepsilon))
              -F(y_\varepsilon,u(y_\varepsilon))\\
       &\leq H\bigl(x_\varepsilon,
              v(z_\varepsilon)-u(y_\varepsilon)\bigr)+o(1),
    \end{align*}
    where the last line follows from the uniform continuity of $F$ and $G$
    on $\overline\Omega\times J$ and
    \eqref{eq-continuous-majorant}.

    Finally let
    \[
       q(X):=
       \begin{cases}
          \psi(X),&X\in B_r(x_0),\\
          w(X),&X\in\R^n\setminus B_r(x_0).
       \end{cases}
    \]
    be the patched test for $w$.  Since
    $h_\varepsilon,\ell_\varepsilon\to0$ and
    $\alpha_\varepsilon+\beta_\varepsilon\to0$, boundedness and translation
    continuity in $L^1_{\rm loc}$ give
    \[
       \int_{\R^n}\frac{|q_\varepsilon-q|}{1+|y|^{n+2s}}\,dy
       \longrightarrow0.
    \]
    Also $q_\varepsilon\to q$ in $C^2$ near $x_0$.  Hence
    Lemma~\ref{lem-patched-stability} yields
    \[
       \M^-q_\varepsilon(x_\varepsilon)\longrightarrow\M^-q(x_0).
    \]
    Passing to the limit gives
    \[
       \M^-q(x_0)\leq H(x_0,w(x_0)),
    \]
    which proves that $w$ is a viscosity supersolution of
    \eqref{eq-difference-semicont}.
\end{proof}

\section{Perron construction without properness}\label{app-perron}

For completeness, we give the proof of Lemma~\ref{lem-perron}.  The
argument follows the Perron envelope and local bump strategy of
\cite[Lemma~3.1 and Theorem~3.2]{mou2017perron}, but it is written directly
for the present translation-invariant Isaacs operator.  In particular, no
properness assumption in the zeroth-order variable is used.

\begin{proof}[Proof of Lemma~\ref{lem-perron}]
    The family $\mathcal F$ is nonempty because $\underline w\in\mathcal F$.
    We write
    \[
       W^*(x):=\limsup_{y\to x}W(y),
       \qquad
       W_*(x):=\liminf_{y\to x}W(y).
    \]
    Since the barriers are continuous,
    \[
       \underline w\leq W_*\leq W^*\leq\overline w
       \quad\text{in }\R^n.
    \]

    We first show that $W^*$ is a viscosity subsolution.  Let
    $\zeta\in C^2(\overline V)$ touch $W^*$ from above at $x_0\in V\Subset
    \Omega$.  As in the proof of Lemma~\ref{lem-difference}, we may assume
    that $W^*-\zeta$ attains a strict local maximum at $x_0$.  Choose $r>0$ with
    $\overline{B_r(x_0)}\Subset V$ and let $q$ be the corresponding patched
    function with exterior branch $W^*$.  By the standard upper-envelope
    selection (cf. \cite[Lemma~3.1]{mou2017perron}), there exist
    $z_j\in\mathcal F$, $x_j\to x_0$, and $c_j\to0$ such that
    $\zeta+c_j$ touches $z_j$ from above at $x_j$ in $B_r(x_0)$ and
    $z_j(x_j)\to W^*(x_0)$.

    Let $p_j$ be the corresponding patched test with exterior branch $z_j$,
    and let $Q_j$ have the same local branch $\zeta+c_j$ and exterior branch
    $W^*$.  Since $z_j\leq W^*$, ellipticity and the subsolution property give
    \[
       IQ_j(x_j)+\mu z_j(x_j)\geq k(x_j)z_j(x_j)^p.
    \]
    Since $Q_j\to q$ locally in $C^2$ near $x_0$ and in weighted $L^1$,
    Lemma~\ref{lem-patched-stability} yields
    \[
       Iq(x_0)+\mu W^*(x_0)\geq k(x_0)W^*(x_0)^p.
    \]
    Thus $W^*$ is a subsolution.  Hence $W^*\in\mathcal F$, and the definition
    of $W$ gives $W^*\leq W$.  Since always $W\leq W^*$, we have
    \begin{equation}\label{eq-perron-usc}
       W=W^*.
    \end{equation}

    We next prove that $W_*$ is a viscosity supersolution.  Suppose, to the
    contrary, that it is not.  Then there exist $x_0\in\Omega$ and a $C^2$
    function $\zeta$ touching $W_*$ from below at $x_0$ such that, for the
    corresponding patched function $q$,
    \begin{equation}\label{eq-perron-strict}
       -Iq(x_0)-\mu W_*(x_0)+k(x_0)W_*(x_0)^p<0.
    \end{equation}
    If necessary, for $\delta>0$ sufficiently small, replace $\zeta$ by
    \[
       \zeta_\delta(x):=\zeta(x)-\delta|x-x_0|^4.
    \]
    Then the contact is strict, while the supersolution inequality remains
    violated.  Shrinking the contact neighborhood if necessary, choose $R>0$
    with $\overline{B_{4R}(x_0)}\Subset\Omega$ such that
    $\zeta\leq W_*$ on $\overline{B_{4R}(x_0)}$, and set
    \[
       q_R(y)=
       \begin{cases}
          \zeta(y),&y\in\overline{B_{4R}(x_0)},\\
          W_*(y),&y\in\R^n\setminus\overline{B_{4R}(x_0)}.
       \end{cases}
    \]
    As $R$ increases, $Iq_R(x_0)$ is nonincreasing.  Hence, by taking $R$
    sufficiently small, \eqref{eq-perron-strict} remains valid with $q$
    replaced by $q_R$.
    The function $q_R$ is bounded, lower semicontinuous, and $C^2$ in
    $B_{4R}(x_0)$.

    Let
    \[
       \vartheta(x):=
       \frac{|x-x_0|^4}{1+|x-x_0|^4}.
    \]
    Choose $d>0$ so small that, with
    \[
       \Phi:=q_R-d\vartheta,
    \]
    Lemma~\ref{lem-patched-stability} gives
    \[
       -I\Phi(x_0)-\mu\Phi(x_0)+k(x_0)\Phi(x_0)^p<0.
    \]
    Since $\Phi$ is $C^2$ in $B_{4R}(x_0)$, Lemma~\ref{lem-patched-stability},
    applied with the function fixed and the evaluation point varying, shows
    that $x\mapsto I\Phi(x)$ is continuous in $B_{3R}(x_0)$.  Since
    $\Phi(x_0)=W_*(x_0)>0$, there are $r\in(0,R)$ and $c_0>0$ such that
    \begin{equation}\label{eq-perron-local-strict}
       -I(\Phi+c)(x)-\mu\bigl(\Phi(x)+c\bigr)
       +k(x)\bigl(\Phi(x)+c\bigr)^p<0
    \end{equation}
    for $x\in\overline{B_{2r}(x_0)}$ and $0\leq c\leq c_0$, and
    $\Phi>0$ on $\overline{B_{2r}(x_0)}$.

    We also have $\overline w(x_0)>W_*(x_0)$.  Indeed, if equality held,
    then $q_R\leq W_*\leq\overline w$ in $\R^n$ and
    $q_R(x_0)=\overline w(x_0)$.  Thus the supersolution inequality for
    $\overline w$, together with ellipticity, would contradict
    \eqref{eq-perron-strict}.  After decreasing $r$ if necessary, choose
    $\sigma>0$ so that
    \begin{equation}\label{eq-perron-upper-gap}
       \Phi+2\sigma\leq\overline w
       \quad\text{in }B_{2r}(x_0).
    \end{equation}
    Finally set
    \[
       d_r:=d\inf_{\R^n\setminus B_{r/2}(x_0)}\vartheta>0.
    \]
    Since $q_R\leq W_*\leq W$ by \eqref{eq-perron-usc},
    \begin{equation}\label{eq-perron-global-gap}
       W\geq\Phi+d_r
       \quad\text{in }\R^n\setminus B_{r/2}(x_0).
    \end{equation}

    Choose $0<c<\min\{c_0,\sigma,d_r\}$ and define
    \[
       U(x)=
       \begin{cases}
          \max\{W(x),\Phi(x)+c\},&x\in B_r(x_0),\\
          W(x),&x\in\R^n\setminus B_r(x_0).
       \end{cases}
    \]
    By \eqref{eq-perron-global-gap}, $U=W$ on
    $B_r(x_0)\setminus B_{r/2}(x_0)$ and in a neighborhood of
    $\partial B_r(x_0)$; hence $U\in USC(\R^n)$.  Moreover,
    $\underline w\leq U\leq\overline w$ by
    \eqref{eq-perron-upper-gap} and the corresponding bounds for $W$.

    We claim that $U$ is a subsolution.  If the active branch at a contact
    point is $W$, then $U\geq W$ globally and equality holds at that point;
    ellipticity transfers the subsolution inequality from $W$ to $U$.  If the
    active branch is $\Phi+c$, let $\psi$ be an upper test for $U$ and let
    $\widetilde\psi$ be its patched function.  By
    \eqref{eq-perron-global-gap}, $U\geq\Phi+c$ in all of $\R^n$, with
    equality at the contact point.  Hence ellipticity gives
    \[
       I\widetilde\psi\geq I(\Phi+c),
    \]
    and \eqref{eq-perron-local-strict} yields the required subsolution
    inequality.  Thus $U\in\mathcal F$.

    Choose $x_j\to x_0$ with $W(x_j)\to W_*(x_0)$.  For all sufficiently
    large $j$, $x_j\in B_{r/2}(x_0)$ and
    \[
       \Phi(x_j)+c>W(x_j),
    \]
    because $\Phi(x_0)=W_*(x_0)$.  Hence $U(x_j)>W(x_j)$, contradicting the
    definition of $W$ as the pointwise supremum of $\mathcal F$.  Thus $W_*$
    is a viscosity supersolution.

    The barriers give the exterior and boundary values and
    $W_*\geq\underline w>0$ in $\Omega$.  Lemma~\ref{lem4} yields
    $W^*\leq W_*$; together with \eqref{eq-perron-usc} and the general
    inequality $W_*\leq W$, this gives
    $W=W^*=W_*\in C(\R^n)$.  Hence $W$ solves \eqref{e1}.
\end{proof}

\section*{Acknowledgments}
The authors wish to thank Sapienza University of Rome for funding S.~Kitano’s visit during May--July 2026, and the Department of Mathematics Guido Castelnuovo for its warm hospitality.
S.~Kitano was partially supported by JSPS KAKENHI Grant-in-Aid for
Early-Career Scientists (JP24K16956) and Grant-in-Aid for JSPS Fellows
(JP25KJ0086).

\section*{Declaration on the use of artificial intelligence}
The original manuscript, including its main results and core mathematical arguments, was prepared entirely by the authors without the use of generative artificial intelligence. During subsequent revisions, the authors used ChatGPT (OpenAI) for iterative mathematical discussion and editorial assistance, including examining, developing, and refining mathematical arguments, checking related literature and references, and improving the exposition. All AI-assisted suggestions and arguments were independently verified, revised, and finalized by the authors, who take full responsibility for the mathematical content of the manuscript.

\bibliographystyle{plain}
\bibliography{references}

\end{document}